\documentclass[a4paper,oneside,11pt]{article}
\usepackage{amsmath}
\usepackage{amssymb}
\usepackage{amsthm}
\usepackage{MnSymbol}
\usepackage{enumerate}
\usepackage{bm}
\usepackage{enumitem}
\usepackage[margin = 3cm]{geometry}
\usepackage{hyperref,url}

\newcommand{\N}{\mathbb{N}}
\newcommand{\R}{\mathbb{R}}
\newcommand{\PP}{\mathbb{P}}
\newcommand{\pd}{\partial}
\newcommand{\U}{\mathcal{U}}
\newcommand{\K}{\mathcal{K}}
\newcommand{\M}{\mathcal{M}}

\newcommand{\Z}{\mathcal{Z}}
\newcommand{\So}{\mathcal{S}}
\newcommand{\pdnu}{\pd_{\bm{\nu}}}
\newcommand{\eps}{\varepsilon}
\newcommand{\e}[1]{#1^{\eps}}
\newcommand{\dd}{\, d}
\newcommand{\dx}{\, dx}
\newcommand{\dt}{\, dt}
\newcommand{\ds}{\, ds}
\newcommand{\dy}{\, dy}
\newcommand{\dz}{\, dz}
\newcommand{\dhaus}{\, d \Gamma}
\newcommand{\EE}{\mathcal{E}}
\newcommand{\GG}{\mathcal{G}}
\newcommand{\HH}{\mathcal{H}}
\newcommand{\NN}{\mathcal{N}}
\newcommand{\der}{\mathrm{D}}
\newcommand{\chara}{\textbf{1}}
\newcommand{\id}{\, \bm{\mathrm{I}}\,}
\newcommand{\abs}[1]{\left| #1 \right|}
\newcommand{\norm}[1]{\| #1 \|}
\newcommand{\bignorm}[1]{\left\| #1 \right\|}
\newcommand{\inner}[2]{\langle #1 , #2 \rangle}
\newcommand{\biginner}[2]{\left\langle #1 , #2 \right\rangle}
\newcommand{\tr}[1]{\mathrm{tr} ( #1 )}
\newcommand{\Laplace}{\Delta}
\newcommand{\mean}[1]{\overline{#1}}
\renewcommand{\div}{\, \mathrm{div}\,}
\newcommand{\jump}[1]{\left [#1 \right ]_{-}^{+}}
\newcommand{\velo}{\mathcal{V}}

\newtheorem{thm}{Theorem}[]
\newtheorem{lem}{Lemma}[section]

\newtheorem{remark}{Remark}[section]

\numberwithin{equation}{section}

\begin{document}

\title{Sharp interface limit of a non-mass-conserving Cahn--Hilliard system with source terms and non-solenoidal Darcy flow}

\author{Kei Fong Lam \footnotemark[1]}

\date{ }

\maketitle

\renewcommand{\thefootnote}{\fnsymbol{footnote}}
\footnotetext[1]{Department of Mathematics, The Chinese University of Hong Kong, Shatin, N.T., Hong Kong
({\tt kflam@math.cuhk.edu.hk}).}

\begin{abstract}
We study the sharp interface limit of a non-mass-conserving Cahn--Hilliard--Darcy system with the weak compactness method developed in Chen (J. Differential Geometry, 1996).  The source term present in the Cahn--Hilliard component is a product of the order parameter and a prescribed function with zero spatial mean, leading to non-conservation of mass.  Furthermore, the divergence of the velocity field is given by another prescribed function with zero spatial mean, which yields a coupling of the Cahn--Hilliard equation to a non-solenoidal Darcy flow.  New difficulties arise in the derivation of uniform estimates due to the presence of the source terms, which can be circumvented if we consider the above specific structures.  Moreover, due to the lack of mass-conservation, the analysis is valid as long as one phase does not vanish completely, leading to a local-in-time result.  The sharp interface limit for a Cahn--Hilliard--Brinkman system is also discussed.
\end{abstract}



\section{Introduction}
The Cahn--Hilliard equation 
\begin{align}\label{Intro:CH}
\pd_{t} \varphi = \Laplace \mu, \quad \mu = - \eps \Laplace \varphi + \eps^{-1} \Psi'(\varphi) 
\end{align}
is a well-known model for spinodal decomposition and has seen many applications in models involving multiple phases of matter \cite{Kim}.  In \eqref{Intro:CH}, $\varphi$ is the order parameter which distinguishes the two phases separated by thin interfacial layers, $\Psi$ is a potential with two equal minima (e.g. at $\pm 1$) with derivative $\Psi'$, $\mu$ is the chemical potential for $\varphi$, and $\eps > 0$ is a small parameter related to the thickness of the interfacial layer.  A key property is the conservation of mass, that is, when \eqref{Intro:CH} is posed in a bounded domain $\Omega \subset \R^{d}$ and furnished with Neumann boundary conditions $\pdnu \varphi := \nabla \varphi \cdot \bm{\nu} = 0$ on $\pd \Omega$, it holds that
\begin{align*}
\int_{\Omega} \varphi(t) \dx = \int_{\Omega} \varphi(0) \dx \quad \forall t > 0,
\end{align*}
where $\varphi(0)$ is the initial condition for $\varphi$.  This property allows for the interpretation that the Cahn--Hilliard equation is the $H^{-1}$-gradient flow of the Ginzburg--Landau functional \cite{Cowan}
\begin{align}\label{GL}
\int_{\Omega} \frac{1}{\eps} \Psi(\varphi) + \frac{\eps}{2} \abs{\nabla \varphi}^{2} \dx.
\end{align}
It is possible to consider material effects for the phases of matter, for instance, if the phases are modelled as solids with different elastic properties.  Then, a model coupling the Cahn--Hilliard equation and a quasistatic linear elasticity system is used \cite{Garcke, GarckeKwak, LarcheCahn, Zhu}.  One may also model the phases as incompressible fluids \cite{Anderson}, leading to coupled systems of Navier--Stokes--Cahn--Hilliard type \cite{AGG, LT} or Cahn--Hilliard--Darcy type \cite{DGL, LLG1, LLG2}.  

Further applications of the Cahn--Hilliard equation are image inpainting \cite{Bertozzi} to fill in damaged or missing regions of an image using data from surrounding areas, and also in the mathematical modelling of tumour growth \cite{CL_book, Frieboes, GLSS, Wise}.  Our motivation stems from the latter application which utilizes a new class of continuum models to capture the basic behaviour of the tumour, such as its growth by consuming a chemical species that serves as a nutrient, and tumour elimination by apoptosis, by affixing source terms in the Cahn--Hilliard equation.  There is additional coupling of the Cahn--Hilliard equation with a reaction-diffusion equation for the chemical species, and/or Darcy-type equation for a velocity field.  Let us mention that in inpainting and tumour growth the Cahn--Hilliard equation is modified in such a way that the mass-conservation property is lost.

In this work we study the following Cahn--Hilliard--Darcy system 
\begin{equation}\label{Intro:CHD}
\begin{aligned}
 \div \bm{v} & = \HH , \\
\K \bm{v} & = -\nabla p + (\mu + \chi \sigma) \nabla \varphi, \\
\pd_{t}\varphi + \div (\varphi \bm{v}) & = \div (m(\varphi) \nabla \mu) + \U \varphi, \\
\mu & = \eps^{-1} \Psi'(\varphi) - \eps \Laplace \varphi - \chi \sigma, \\
\pd_{t}\sigma + \div (\sigma \bm{v}) & = \div (n(\varphi) \nabla (\sigma - \chi \varphi)) + \So.
\end{aligned}
\end{equation}
The above system is composed of a convective Cahn--Hilliard equation with source term $\U \varphi$ and variable mobility $m(\varphi)$ that is coupled to an equation for the variable $\sigma$ with source term $\So$ and variable mobility $n(\sigma)$ through the fluxes.  For positive values of $\chi$, the appearance of $\div (\chi m(\varphi) \nabla \sigma)$ in the equation for $\varphi$, after substituting the definition of $\mu$, and of $\div (\chi n(\varphi) \nabla \varphi)$ in the equation for $\sigma$ lead to a type of cross-diffusion coupling between $\varphi$ and $\sigma$, and can be used to model chemotaxis and active transport mechanisms, respectively \cite{GLNeu, GLSS}.  Both variables $\varphi$ and $\sigma$ are transported by a volume-averaged mixture velocity $\bm{v}$ \cite{AGG} which is governed by Darcy's law with pressure $p$ and inverse permeability $\K$.  The term $\HH$ accounts for  sources and sinks in the mass balance due to the gain or loss of volume from the term $\U \varphi$.  This system of equations arises for example when modelling tumour and healthy cells as inertia-less fluids \cite{ChenLowen, Frieboes, GLSS, Wise} and its variants have been previously studied analytically \cite{GLDarcy, JWZ, LTZ}.  Note that the velocity $\bm{v}$ is non-solenoidal if $\HH$ is non-zero, and in the presence of the source terms $\HH$, $\U \varphi$, and $\So$, the system \eqref{Intro:CHD} admits an energy identity of the form
\begin{equation}\label{Intro:Energy:ID}
\begin{aligned}
& \frac{\dd}{\dt} \int_{\Omega} \frac{1}{\eps} \Psi(\varphi) + \frac{\eps}{2} \abs{\nabla \varphi}^{2} + \frac{1}{2} \abs{\sigma}^{2} - \chi \sigma \varphi \dx \\
& \qquad + \int_{\Omega} m(\varphi) \abs{\nabla \mu}^{2} + n(\varphi) \abs{\nabla (\sigma - \chi \varphi)}^{2} + \K \abs{\bm{v}}^{2} \dx  \\
& \quad = \int_{\Omega} \U \varphi \mu + \So (\sigma - \chi \varphi) + \HH \left ( p - \varphi \mu - \tfrac{1}{2} \abs{\sigma}^{2} \right ) \dx.
\end{aligned}
\end{equation}
Note that the cross-diffusion mechanisms accounting for chemotaxis and active transport leads to a term $\chi \sigma \varphi$ without a definite sign under the time derivative.  Moreover, for non-solenoidal flow, the pressure $p$ appears explicitly in the energy identity.  Also of interest is the Cahn--Hilliard--Brinkman system \cite{BCG}, which modifies the Darcy law $\K \bm{v} = -\nabla p + (\mu + \chi \sigma) \nabla \varphi$ to the following Brinkman equation \cite{Brinkman} with capillary effects:
\begin{align*}
- \div (2 \eta \der \bm{v}) + \K \bm{v} = -\nabla p + (\mu + \chi \sigma) \nabla \varphi,
\end{align*}
where the constant $\eta > 0$ can be interpreted as the viscosity, and $\der \bm{v} := \frac{1}{2} (\nabla \bm{v} + (\nabla \bm{v})^{\top})$ is the rate of deformation tensor.  

The main goal of this work is to establish the behaviour of solutions in the singular limit (also known as the sharp interface limit) in which the thickness of the interfacial layer $\eps$ tends to zero.  Heuristically, in the limit $\eps \to 0$, the domain $\Omega$ is partitioned into two sets: $\Omega_{+}$ where $\varphi$ is close to $+1$ and $\Omega_{-}$ where $\varphi$ is close to $-1$.  These sets are separated by a moving hypersurface $\Gamma$ whose evolution is coupled to partial differential equations posed in $\Omega_{\pm}$, which results in a free boundary problem.  The formal identification of sharp interface limits for phase field models in great generality can be obtained rather easily with the method of matched asymptotic expansions \cite{Fife, Pego}, but the rigorous passage $(\eps \to 0)$ have only been showed for a handful of models.  For the Cahn--Hilliard equation and its variants we mention three methodologies for the rigorous passage.

The first method constructs approximate solutions to the phase field model via asymptotic expansions.  The leading zeroth order term of the expansion is built from smooth solutions of the limiting free boundary problem, while the subsequent terms are obtained from rigorously solving the systems of equations arising from matching expansions order-by-order.  The error between the approximate solutions and the true solution to the phase field model is controlled by spectral estimates of the Cahn--Hilliard operator \cite{ChenSpectral}, and the number of terms in the expansion depends on a desired decay rate for the error.  This method has been applied to the Cahn--Hilliard equation \cite{ABC}, the Cahn--Larch\'{e} system \cite{AbelsSchaubeck1}, the convective Cahn--Hilliard equation \cite{AbelsSchaubeck2}, and a diffuse interface model for tumour growth \cite{Fei2}.  Let us also mention a related method performed with Hilbert expansions  \cite{CCO} and the convergence of fully discrete numerical solutions of the Cahn--Hilliard equation \cite{FengCH}.  While this method of rigorous asymptotic expansions can yield a convergence rate in $\eps$, it is a \emph{local-in-time} result, as it is expected that the free boundary problem admits only local-in-time smooth solutions.  Furthermore, the construction of approximate solutions can become cumbersome rather quickly and good decay estimates are needed for the error terms.

The second method utilises the well-known results of Modica--Mortola \cite{Modica} on the Gamma-convergence of the Ginzburg--Landau functional \eqref{GL} to a constant multiple of the perimeter functional, and the work of Sandier and Serfaty \cite{SS} on the Gamma-convergence of gradient flows.  Le \cite{Le} took advantage of the fact that the Cahn--Hilliard equation is the $H^{-1}$-gradient flow of the Ginzburg--Landau functional \eqref{GL} in order to pass to the limit $\eps \to 0$.  However, this method is only limited to \emph{gradient flows}, and thus source terms in the Cahn--Hilliard system have to be formulated in such a way that the system can be expressed as a gradient flow of some functional \cite{RoccaScala}.  In addition, this method requires a priori smoothness on the limit hypersurface $\Gamma$ in the limiting free boundary problem ($C^{3}$-regularity) and an $H^{1}$-variant of a De Giorgi conjecture to hold. 

The third method seeks uniform bounds in $\eps$ for the variables in appropriate function spaces and passing to the limit $\eps \to 0$ with weak compactness methods.  One obtains as a consequence the existence of a weak solution to the free boundary problem, where the hypersurface $\Gamma$ is represented as a varifold, and the mean curvature is formulated as the first variation of the varifold.  The convergence holds global-in-time, but unlike the first method, no convergence rate in $\eps$ can be deduced.  This method has been used for the Cahn--Hilliard equation in the seminal work of Chen \cite{Chen96} (cf. Stoth \cite{Stoth} for the radial setting), and have seen successful applications to the Navier--Stokes--Cahn--Hilliard system \cite{AbelsLengeler, AbelsRoger}, the Cahn--Larch\'{e} system \cite{GarckeKwak}, the Hele--Shaw--Cahn--Hilliard system \cite{Fei,RM}, and a model for lipid raft formation with a surface Cahn--Hilliard equation \cite{AK}.  It is worth pointing out that a uniform bound on the mean value of the chemical potential $\mu$ can be obtained if the mean of $\varphi$ stayed uniformly away from $-1$ and $1$.  One way to guarantee this is to enforce that the mean of $\varphi(0)$ lies in $(-1,1)$ and use the mass-conservation property of the Cahn--Hilliard equation.  Thus, the references mentioned above for this weak compactness approach do not deal with a source term in the Cahn--Hilliard equation, which would automatically destroy the mass-conservation property.  The main purpose of this work is to provide a result in this direction.

Let us illustrate the main obstacle in deriving uniform estimates even for the Cahn--Hilliard equation with a bounded source term $\Z$, that is, 
\begin{align*}
\pd_{t} \varphi = \Laplace \mu + \Z, \quad \e{\mu} = - \eps \Laplace \varphi + \eps^{-1} \Psi'(\varphi),
\end{align*} 
whose solutions satisfy the energy identity
\begin{align}\label{CH:source}
\frac{d}{dt} \int_{\Omega} \frac{1}{\eps} \Psi(\varphi) + \frac{\eps}{2} \abs{\nabla \varphi}^{2} \dx + \int_{\Omega} \abs{\nabla \mu}^{2} \dx = \int_{\Omega} \Z \mu \dx.
\end{align}
A key observation to control the right-hand side is to employ a splitting and the Poincar\'{e} inequality, leading to
\begin{align*}
\int_{\Omega} \Z \mu \dx \ds & = \int_{\Omega} \Z (\mu - \mean{\mu}) \dx + \mean{\mu} \int_{\Omega} \Z \dx \leq C \norm{\nabla \mu}_{L^{2}(\Omega)} + C \abs{\mean{\mu}},
\end{align*}
where $\mean{\mu} := \frac{1}{\abs{\Omega}} \int_{\Omega} \mu \dx$ is the spatial mean of $\mu$.  Hence, uniform estimates can be obtained from \eqref{CH:source} if the spatial mean $\mean{\mu}$ can be bounded uniformly in $\eps$, or bounded by the terms on the left-hand side of \eqref{CH:source} so that a Gronwall argument can be applied.  The method of controlling $\mean{\mu}$ established in the seminal work of Chen \cite{Chen96} relies on $\int_{\Omega} \eps^{-1} \Psi(\varphi) \dx$ is uniformly bounded beforehand.  This is the main difference between previous works \cite{AbelsLengeler, AbelsRoger, Chen96, Fei, GarckeKwak, RM} and the current work, where we need to derive an estimate on the mean value $\mean{\mu}$  absent of any uniform estimates.  We overcome this technical difficult by considering source term $\Z$ of the form $\U \varphi$ for a prescribed function $\U$ with zero spatial mean, which allows us to adapt the proof of \cite[Lem.~3.4]{Chen96} to derive the first uniform estimate from the energy identity \eqref{Intro:Energy:ID}, and subsequent uniform estimates then follow.  On the other hand, observe that under natural boundary conditions we do not necessarily have mass conservation for $\varphi$:
\begin{align}\label{mean:varphi}
\frac{1}{\abs{\Omega}} \int_{\Omega} \varphi(t) \dx - \frac{1}{\abs{\Omega}} \int_{\Omega} \varphi_{0} \dx = \int_{0}^{t} \frac{1}{\abs{\Omega}} \int_{\Omega} \U \varphi \dx \ds,
\end{align}
as the right-hand side of \eqref{mean:varphi} need not be zero.  To employ the weak compactness machinery developed by Chen \cite{Chen96} and analyse the sharp interface limit of \eqref{Intro:CHD}, we have to estimate the right-hand side of \eqref{mean:varphi}, and thus place an upper bound on the terminal time $T$ for which the mean of $\varphi(t)$ stays away from $-1$ and $1$ for all $t \in [0,T]$.  Physically, this means that in the time interval $[0,T]$ under consideration, neither phases of matter can vanish completely.   An upper bound on the terminal time means the nature of our sharp interface analysis is \emph{local-in-time}, as oppose to previous global-in-time results in the literature.  

An interesting feature of the term $\div (n(\varphi) \chi \nabla \varphi)$ in equation for $\sigma$, as revealed in formally matched asymptotic analysis and supported by numerical simulations \cite{GLSS}, is that for $\chi > 0$, $\sigma$ will experience a jump across the hypersurface $\Gamma$.  More precisely,
\begin{align}\label{Intro:jump}
\jump{\sigma} = 2 \chi \text{ on } \Gamma,
\end{align}
where $\jump{f}(x) := \lim_{\delta \to 0} (f(x + \delta \nu(x)) - f(x - \delta \nu(x)) )$ denotes the jump of the quantity $f$ across $\Gamma$.  Here, $x \in \Gamma$ and $\nu$ is the normal vector field on $\Gamma$ pointing into $\Omega_{+}$, with points $x + \delta \nu(x) \in \Omega_{+}$ and $x - \delta \nu(x) \in \Omega_{-}$.  For physically relevant situations, $\sigma$ represents the density of a chemical in the system, and should always be non-negative.  In the case $\chi = 0$, such a qualitative behaviour can be proved with the help of a weak comparison principle, and boundedness may even be shown if the initial condition is bounded.  But it turns out that the same arguments cannot be applied when $\chi > 0$, and establishing an $L^{\infty}(0,T;L^{\infty}(\Omega))$ estimate for $\sigma$ by the Moser--Alikakos iteration \cite{Alikakos} seems not possible due to the cross-diffusion type term $\div (n(\varphi) \chi \nabla \varphi)$.  Even if $\varphi \in L^{2}(0,T;H^{3}(\Omega))$ - which is typical for weak solutions to Cahn--Hilliard systems, preliminary calculations suggest at best one obtains $\sigma \in L^{\infty}(0,T;L^{q}(\Omega))$ for all $q \in (2,\infty)$ with the estimates dependent on the value of $q$, and thus at present no qualitative statements can be made about $\sigma$ for positive values of $\chi$.  In fact, we have observed in numerical simulations that negative values of $\sigma$ can appear if $\chi$ is chosen inappropriately.

Hence, a secondary goal of this paper is to ascertain that the jump condition \eqref{Intro:jump} is present in the limit $\eps \to 0$.  Then, in the sharp interface limit, $\sigma$ satisfies second order equations in the bulk regions $\Omega_{-}$ and $\Omega_{+}$, for which certain comparison principles may be available, and the jump condition \eqref{Intro:jump} can give an indication of the appropriate range of values for $\chi$ to choose to achieve physically relevant numerical simulations.

The outline of the paper is as follows: In Sec.~\ref{sec:Prelim} we introduce the notation and several preliminary results that will be useful in studying the sharp interface limit.  The main result is stated in Sec.~\ref{sec:Main}.  We derive uniform estimates in Sec.~\ref{sec:Est}, compactness statements in Sec.~\ref{sec:Comp}, and then pass to the limit in Sec.~\ref{sec:pass}.  In Sec.~\ref{sec:diss} we discuss analogous sharp interface limits for the zero-velocity variant and the solenoidal Brinkman variant.

\section{Useful preliminaries}\label{sec:Prelim}

\subsection{Notation}

The identity matrix will be denoted by $\id$.  For any two vectors $\bm{a}, \bm{b} \in \R^{d}$, the tensor product $\bm{a} \otimes \bm{b} \in \R^{d \times d}$ is defined as $(\bm{a} \otimes \bm{b})_{ij} = a_{i} b_{j}$ for $1 \leq i, j \leq d$.  The Jacobian matrix of a differentiable vector-valued function $\bm{Y} : \Omega \to \R^{d}$ is denoted as $\nabla \bm{Y} \in \R^{d \times d}$, and the Hessian of a twice-differentiable scalar function $\psi: \Omega \to \R$ is denoted as $D^{2} \psi$.  For the sharp interface limits, we obtain a decomposition of the domain $\Omega$ into two disjoint time-dependent subsets $\Omega_{+}$ and $\Omega_{-}$ separated by a time-dependent interface $\Gamma$.  The jump of a function $f$ across $\Gamma$ is defined as
\begin{align*}
\jump{f}(x) := \lim_{\delta \to 0} \left ( f(x + \delta \nu(x)) - f(x - \delta \nu(x)) \right ),
\end{align*}
for $x \in \Gamma$, with a normal vector $\nu$ of $\Gamma$ pointing to $\Omega_{+}$ such that $x + \delta \nu(x) \in \Omega_{+}$ and $x - \delta \nu(x) \in \Omega_{-}$.  Furthermore, we will denote the normal velocity of the interface $\Gamma$ by $\velo$, and the characteristic function of any measurable set $A$ by $\chara_{A}$.


\paragraph{Function spaces.}
We use the notation $L^{p} := L^{p}(\Omega)$ and $W^{k,p} := W^{k,p}(\Omega)$ for any $p \in [1,\infty]$, $k > 0$ to denote the standard Lebesgue spaces and Sobolev spaces equipped with the norms $\norm{\cdot}_{L^{p}}$ and $\norm{\cdot}_{W^{k,p}}$.  In the case $p = 2$ we use notation $H^k(\Omega) = W^{k,2}(\Omega)$ and $\norm{\cdot}_{H^{k}} := \norm{\cdot}_{W^{k,2}}$.  The space-time cylinder $\Omega \times (0,T)$ is denoted as $Q$, and for any $t \in (0,T)$, we denote $\Omega \times (0,t)$ as $Q_{t}$.  Similarly, we denote $\pd \Omega \times (0,T)$ by $\Sigma$.  For any Banach space $X$, its dual is denoted as $X'$, and for any $p \in [1,\infty]$, the norm of the Bochner space $L^{p}(0,T;X)$ will sometimes be denoted as $\norm{\cdot}_{L^{p}(X)}$.  We will often use the isometric isomorphism $L^{p}(Q_{t}) \cong L^{p}(0,t;L^{p})$ for $t \in (0,T]$ and any $p \in [1,\infty)$.  Furthermore, we use the notation $L^{p}(Q) := L^{p}(Q_{T})$ and $(L^{p}(Q))^{d}$ denotes the space of vector-valued functions $\bm{f} :Q \to \R^{d}$ where each component belongs to $L^{p}(Q)$.  The space $L^{p}(Q)^{d \times d}$ is defined analogously for tensor-valued functions.

The mean of a function $f$ will be denoted as $\mean{f} := \frac{1}{\abs{\Omega}} \int_{\Omega} f \dx$.  We introduce the space $L^{p}_{0} := \{ f \in L^{p}(\Omega) : \int_{\Omega} f \dx = 0 \}$, $p \in [1,\infty]$, as the space of $L^{p}(\Omega)$-functions with zero mean.  Moreover, we define $H^2_n := \{ f \in H^{2}(\Omega) : \pdnu f = 0 \text{ on } \pd \Omega \}$, where $\pdnu f := \nabla f \cdot \bm{\nu}$ is the normal derivative of $f$ on $\pd \Omega$.

For $k \in \N \cup \{\infty\}$, we denote by $C^{k}_{c}(\Omega)$ the space of $C^{k}$-functions with compact support in $\Omega$, and by $C^{k}_{0}(\Omega)$ the space of $C^{k}$-functions that have zero trace on $\pd \Omega$.  For a Banach space $B$ and $\alpha \in (0,1)$, we denote the space
$C^{0,\alpha}([0,T];B)$ as the set of functions $f \in C^{0}([0,T];B)$ such that $\norm{f}_{C^{0,\alpha}([0,T];B)} < \infty$ where
\begin{align*}
\norm{f}_{C^{0,\alpha}([0,T];B)} & := \sup_{t \in [0,T]} \norm{f(t)}_{B} + \sup_{0 \leq \tau < t \leq T} \left ( [f]_{\alpha, B}(\tau, t) \right ), \\
[f]_{\alpha,B}(\tau,t) & := \frac{\norm{f(t) - f(\tau)}_{B}}{\abs{t-\tau}^{\alpha}}.
\end{align*}


\paragraph{Weakly continuous functions.}
For a Banach space $X$ with dual space $X'$, we denote the space $C^{0}_{\mathrm{w}}([0,T];X)$ as the set of functions $f: [0,T] \to X$ such that $t \mapsto \norm{f(t)}_{X}$ is bounded and the mapping $t \mapsto \inner{\phi}{f(t)}_{X}$ is continuous on $[0,T]$ for all $\phi \in X'$, where $\inner{\cdot}{\cdot}_{X}$ denotes the duality pairing between $X$ and its dual.  Furthermore, we say that $f_{n} \to f$ strongly in $C^{0}_{\mathrm{w}}([0,T];X)$ if and only if
\begin{align*}
\inner{\phi}{f_{n}(t)}_{X} \to \inner{\phi}{f(t)}_{X} \text{ in } C^{0}([0,T]) \quad \forall \phi \in X'.
\end{align*}


\paragraph{Measures and total variation.}
For a locally compact, separable metric space $Z$ with Borel $\sigma$-algebra $\mathcal{B}(Z)$, let $C^{0}_{c}(Z; \R^{d})$ denote the closure of compactly supported continuous functions $f : Z \to \R^{d}$ with respect to the supremum norm.  The space of $\R^{d}$-valued finite Radon measures, denoted by $\M(Z; \R^{d})$ can be seen as the dual space of $C^{0}_{c}(Z; \R^{d})$ \cite[Thm. 1.54]{Ambrosio}, that is, $\lambda \in \M(Z; \R^{d}) = (C^{0}_{c}(Z; \R^{d}))'$ if and only if $\lambda[\bm{v}_{1} + \bm{v}_{2}] = \lambda[\bm{v}_{1}] + \lambda[\bm{v}_{2}]$ for all $\bm{v}_{1}, \bm{v}_{2} \in C^{0}_{c}(Z; \R^{d})$ and
\begin{align*}
\norm{\lambda} := \sup \{ \lambda[\bm{v}] : \bm{v} \in C^{0}_{c}(Z; \R^{d}), \norm{\bm{v}}_{L^{\infty}(Z)} \leq 1 \} < \infty.  
\end{align*}
We use the notation $\M(Z) := \M(Z; \R)$.  For $\lambda \in \M(Z; \R^{d})$, its total variation, denoted by $\abs{\lambda}$, is given as
\begin{align*}
\abs{\lambda}(B) = \sup \left \{ \sum_{j=1}^{\infty} \abs{\lambda(B_{j})} : B_{j} \in \mathcal{B}(Z) \text{ pairwise disjoint, } B = \bigcup_{j=1}^{\infty} B_{j} \right \}
\end{align*}
for every $B \in \mathcal{B}(Z)$, where $\abs{\lambda(B_{j})}$ denotes the measure of $B_{j}$ with respect to $\lambda$.  For an open set $\Omega \subset \R^{d}$, the set $BV(\Omega) : = \{ f \in L^{1}(\Omega) : \nabla f \in \M(\Omega; \R^{d}) \}$ denotes the space of functions of bounded variations.  The total variation of $f \in BV(\Omega)$ is defined as
\begin{align*}
\norm{\nabla u}_{\M^{d}} := \abs{\nabla u}(\Omega) := \sup \left \{ \int_{\Omega} u \div \bm{Y} \dx \, : \, \bm{Y} \in C^{1}_{c}(\Omega; \R^{d}), \norm{\bm{Y}}_{L^{\infty}} \leq 1 \right \},
\end{align*} 
and we equipped $BV(\Omega)$ with the norm $\norm{f}_{BV} := \norm{f}_{L^{1}} + \norm{\nabla f}_{\M^{d}}$.  Moreover, the space $BV(\Omega, \{0,1\})$ denotes the set of all $f \in BV(\Omega)$ such that $f(x) \in \{0,1\}$ for a.e.~$x \in \Omega$.  It is well-known that $W^{1,1}(\Omega)$ embeds continuously into $BV(\Omega)$ with $\norm{\nabla u}_{L^{1}} = \norm{\nabla u}_{\M^{d}} = \abs{\nabla u}(\Omega)$.


\paragraph{Varifolds.}
A general $(d-1)$-varifold $V$ on $\Omega \subset \R^{d}$ is a Radon measure on the Grassmannian manifold $G(\overline{\Omega}) := \overline{\Omega} \times \PP^{d-1}$, where $\PP^{d-1}$ is the set of $(d-1)$-dimensional subspaces of $\R^{d}$ and can be identified as with $\mathbb{S}^{d-1} / \{ \nu, - \nu\}$, i.e., the set of all normals modulo orientation.  This allows us to associate the $(d-1)$-dimensional subspace $P \in \PP^{d-1}$ with the normal vector in $\mathbb{S}^{d-1} / \{ \nu, - \nu \}$ perpendicular to $P$.  Henceforth, we will reuse the symbol $P$ to denote an element in $\mathbb{S}^{d-1}/ \{\nu, -\nu\}$.  The mass measure $\norm{V}$ of the varifold $V$ is a Radon measure on $\Omega$ defined by 
\begin{align*}
\norm{V}(A) = \int_{A \times \PP^{d-1}} \dd V(x, P) \text{ for } A \subset \Omega.
\end{align*}
Furthermore, its first variation, denoted by $\delta V$, is a linear functional on $C^{1}_{0}(\Omega; \R^{d})$ defined as 
\begin{align}\label{1stVar}
\inner{\delta V}{\bm{Y}} := \int_{\Omega \times \PP^{d-1}} \nabla \bm{Y} : \left ( \id - P \otimes P \right ) \dd V(x, P) \quad \forall \bm{Y} \in C^{1}_{0}(\Omega ; \R^{d}).
\end{align}
A varifold $V$ is said to have a generalised mean curvature $\bm{H}$ if $\bm{H}$ is a $\norm{V}$-measurable vector-valued function satisfying
\begin{align*}
\inner{\delta V}{\bm{Y}} = - \int_{\Omega} \bm{Y} \cdot \bm{H} \dd \norm{V}(x) \quad \forall \bm{Y} \in C^{1}_{0}(\Omega; \R^{d}).
\end{align*}

\subsection{Assumptions and auxiliary results}

Throughout this paper, the following assumptions hold unless stated otherwise.

\begin{enumerate}[label=$(\mathrm{A \arabic*})$, ref = $\mathrm{A \arabic*}$]
\item $\Omega \subset \R^{d}$, $d = 1,2,3$, is a bounded domain with smooth boundary $\pd \Omega$.  The constants $\chi$ and $\K$ are positive.

\item \label{ass:mn} The functions $m, n \in C^{0}(\R)$ satisfy
\begin{align*}
m_{0} \leq m(s) \leq m_{1}, \quad n_{0} \leq n(s) \leq n_{1} \quad \forall s \in \R 
\end{align*} 
with positive constants $m_{0}$, $m_{1}$, $n_{0}$ and $n_{1}$.
\item \label{ass:Psi} $\Psi \in C^{3}(\R)$ is a non-negative function with minima at $\pm 1$ and satisfies
\begin{align}\label{Psi:basic}
\int_{-1}^{1} \sqrt{2 \Psi(s)} \ds = 1, \quad \Psi'(\pm 1) = 0, \quad \Psi(y) \leq 1 + \abs{y}^{2}  \quad \forall y \in [-1,1], 
\end{align} 
and
\begin{align}\label{Psi''lowergrowth}
\exists c_{0} > 0  \text{ s.t. } \Psi''(y) \geq c_{0} \abs{y}^{q-2} \quad \forall \abs{y} > 1 - c_{0}, 
\end{align}
for some exponent $q \geq 4$.  As a consequence, there exist positive constants $k_{0}$, $k_{1}$ such that
\begin{align}
\label{Psi:lowerbound}
\Psi(y) \geq k_{0} \abs{y}^{q} - k_{1} &\quad \forall \abs{y} > 1-c_{0}.
\end{align}
\item \label{ass:initial} For all $\eps \in (0,1]$, $\e{\varphi}_{0} \in H^{1}(\Omega), \e{\sigma}_{0} \in L^{2}(\Omega)$, and there exists a positive constant $\EE_{0}$ and a constant $u_{0} \in (-1,1)$, independent of $\eps$, such that
\begin{align*}
\mean{\e{\varphi}_{0}} = u_{0}, \quad \EE(\e{\varphi}_{0}, \e{\sigma}_{0}) := \int_{\Omega} \frac{1}{\eps} \Psi(\e{\varphi}_{0}) + \frac{\eps}{2} \abs{\nabla \e{\varphi}_{0}}^{2} + \frac{1}{2} \abs{\e{\sigma}_{0}}^{2} - \chi \e{\sigma}_{0} \e{\varphi}_{0} \dx \leq \EE_{0}.
\end{align*}
Furthermore, there exist $\Omega_{+}(0) \subset \subset \Omega$ with $\chara_{\Omega_{+}(0)} \in BV(\Omega, \{0,1\})$ and $\sigma_{0} \in L^{2}(\Omega)$ such that $\e{\sigma}_{0} \to \sigma_{0}$ weakly in $L^{2}(\Omega)$ and $\e{\varphi}_{0} \to -1 + 2 \chara_{\Omega_{+}(0)}$ strongly in $L^{2}(\Omega)$.
\item \label{ass:source} $\U(x) \in C^{0,1}(\overline{\Omega}) \cap L^{\infty}_{0}(\Omega)$ is given, and $T > 0$ is a fixed constant satisfying
\begin{align}\label{cond:T}
\abs{u_{0}} + T \norm{\U}_{L^{\infty}(\Omega)} < 1,
\end{align}
where $u_{0} \in (-1,1)$ is the mean of the initial condition $\e{\varphi}_{0}$.  Furthermore, $\So(x,t)$ and $\HH(x,t)$ are prescribed functions satisfying $\So \in L^{2}(0,T;L^{2}(\Omega))$, $\HH \in L^{\infty}(0,T;L^{\infty}_{0}(\Omega))$. 
\end{enumerate}

The value $\int_{-1}^{1} \sqrt{2 \Psi(s)} \ds$ can be seen as the surface tension which we rescale to $1$, and the other assumptions in \eqref{Psi:basic}, along with \eqref{Psi''lowergrowth}, will be used to show the assertions in Lemmas~\ref{lem:Psigeqabsvarphi-1square} and \ref{lem:ModicaAnsatz:BijW} below.  \eqref{ass:mn} implies the mobilities are non-degenerate, while \eqref{ass:initial}-\eqref{ass:source} allow us to deduce the first uniform estimate, which other uniform estimates then follow from.  The condition \eqref{cond:T} places an upper bound on the time interval $[0,T]$ for which the sharp interface analysis is valid.  In particular, under \eqref{cond:T} we show that the mean satisfies $\mean{\e{\varphi}(t)} \in (-1,1)$ for all $t \in [0,T]$.  Furthermore, compared to the Cahn--Hilliard case \cite{Chen96} the stronger assumption $q \geq 4$ in \eqref{Psi''lowergrowth} is need to derive uniform estimates for the analysis of the Cahn--Hilliard--Darcy system.

We now present some preliminary results that will be crucial in proving our main results.  While they have been used extensively in the literature \cite{AbelsLengeler, AbelsRoger, Chen96, Fei, GarckeKwak, RM}, we are unable to find a reference (aside from similar results in Daube \cite{Daube}).  Due to their relative importance in the study of the sharp interface limit, we sketch the proofs for the benefit of the reader.

\begin{lem}\label{lem:Psigeqabsvarphi-1square}
There exists a positive constant $C_{0}$ such that
\begin{align}\label{quadraticbound}
f(y) := (\abs{y}-1)^{2} \leq C_{0} \Psi(y) \quad \forall y \in \R.
\end{align}
\end{lem}

\begin{proof}
For $\abs{y} \geq 1 - c_{0}$, let $C_{1}$ be a constant such that $C_{1} c_{0}(1 - c_{0})^{q-2} \geq 2$. where $q \geq 4$ is the exponent in \eqref{Psi''lowergrowth}.  Then, $C_{1}\Psi''(y) \geq 2 = f''(y)$ for all $\abs{y} \geq 1 - c_{0}$.  Using that $\Psi(\pm 1) = f(\pm 1 ) = \Psi'(\pm 1) = f'(\pm 1 ) = 0$, we obtain from the fundamental theorem of calculus,
\begin{align*}
C_{1}\Psi'(y) \geq f'(y), \quad C_{1} \Psi(y) \geq f(y) \quad \forall \abs{y} \geq 1 - c_{0}.
\end{align*}
For $\abs{y} \leq 1 - c_{0}$, let $C_{2}$ be a constant such that $C_{2} = 1 / (\min_{\abs{y} \leq 1 - c_{0}} \Psi(y))$.  By assumption $\Psi(y) > 0$ for $\abs{y} \leq 1 - c_{0}$, and so $C_{2}$ is finite.  Then, as $ f(y) \leq 1$ for $\abs{y} \leq 1 - c_{0}$, we have that
\begin{align*}
C_{2} \Psi(y) \geq 1 \geq f(y) \quad \forall \abs{y} \leq 1 - c_{0}.
\end{align*}
The desired constant $C_{0}$ can be taken as $C_{0} = \max (C_{1}, C_{2})$.
\end{proof}

\begin{lem}\label{lem:ModicaAnsatz:BijW}
The function
\begin{align*}
W(y) := \int_{-1}^{y}  \sqrt{2\tilde{\Psi}(s)} \ds \text{ where } \tilde{\Psi}(s) := \min \left ( \Psi(s), 1 + \abs{s}^{2} \right )
\end{align*}
is bijective and there exists a positive constant $C_{1}$ such that for all $y_{1}, y_{2} \in \R$,
\begin{align}\label{Wfunctiondiff}
C_{1} \abs{y_{1} - y_{2}}^{2} \leq \abs{W(y_{1}) - W(y_{2})} \leq \sqrt{2} \abs{y_{1} - y_{2}}(1 + \abs{y_{1}} + \abs{y_{2}}).
\end{align}
\end{lem}

\begin{proof}
For the second inequality of \eqref{Wfunctiondiff}, using $(1 + \abs{s}^{2})^{\frac{1}{2}} \leq (1 + \abs{s})$ yields
\begin{align*}
\abs{W(y_{1}) - W(y_{2})} & =  \abs{\int_{y_{2}}^{y_{1}} \sqrt{2 \tilde{\Psi}(s)} \ds} \leq \sqrt{2} \int_{y_{2}}^{y_{1}} \sqrt{1 + \abs{s}^{2}} \ds \\
&  \leq  \sqrt{2} \int_{y_{2}}^{y_{1}} (1 + \abs{s}) \ds  \leq \sqrt{2} (1 + \abs{y_{1}} + \abs{y_{2}}) \int_{y_{2}}^{y_{1}} 1 \ds.
\end{align*}
Meanwhile, for the first inequality, we claim that $\tilde{\Psi}(y) \geq \min (C_{0}^{-1}, 1) (\abs{y} - 1)^{2}$, where $C_{0}$ is the constant in \eqref{quadraticbound}.  Indeed, from \eqref{quadraticbound}, $\Psi(y) \geq C_{0}^{-1} (\abs{y} - 1)^{2}$, and so 
\begin{align*}
\tilde{\Psi}(y) \geq \min (C_{0}^{-1} (\abs{y} - 1)^{2}, 1 + \abs{y}^{2}).
\end{align*}
If $C_{0} \geq 1$, we further obtain $\tilde{\Psi}(y) \geq C_{0}^{-1} \min ( (\abs{y} - 1)^{2}, 1 + \abs{y}^{2}) = C_{0}^{-1}(\abs{y} - 1)^{2}$, and if $C_{0} \leq 1$, then 
\begin{align*}
\tilde{\Psi}(y) \geq \min (C_{0}^{-1}(\abs{y} - 1)^{2}, 1 + \abs{y}^{2}) \geq \min ((\abs{y} - 1)^{2}, 1 + \abs{y}^{2}) = (\abs{y} - 1)^{2}.
\end{align*}
This yields the claim.  Then, we see that
\begin{align*}
\abs{W(y_{1}) - W(y_{2})} = \abs{\int_{y_{2}}^{y_{1}} \sqrt{2 \tilde{\Psi}(y)} \ds} \geq  \sqrt{2} \min(C_{0}^{-1},1) \int_{y_{2}}^{y_{1}} \abs{\abs{s} - 1} \ds.
\end{align*}
Without loss of generality, we assume $y_{2} < y_{1}$, and a case analysis will yield the first inequality of \eqref{Wfunctiondiff}.  There are five cases to consider:
\begin{enumerate}
\item[(1)] $(1 \leq y_{2} < y_{1})$, we have
\begin{align*}
2 \int_{y_{2}}^{y_{1}} \abs{\abs{s} - 1} \ds & = 2 \int_{y_{2}}^{y_{1}} s-1 \ds = ((y_{1} - 1)^{2} - (y_{2} - 1)^{2}) \\
& =  (y_{1} - y_{2})(y_{1} - y_{2} + 2 y_{2} - 2) \geq \abs{y_{1} - y_{2}}^{2}.
\end{align*}
\item[(2)] $(0 \leq y_{2} < y_{1} \leq 1)$, a similar computation to case (1) yields
\begin{align*}
2 \int_{y_{2}}^{y_{1}} \abs{\abs{s} - 1} \ds & = 2 \int_{y_{2}}^{y_{1}} 1 - s \ds = ((y_{2} - 1)^{2} - (y_{1} - 1)^{2})\\
&  = (y_{2} - y_{1}) (y_{2} - y_{1} + 2 y_{1} - 2) \geq \abs{y_{1} - y_{2}}^{2}.
\end{align*}
\item[(3)] $(0 \leq y_{2} \leq 1 < y_{1})$, using triangle inequality and Young's inequality leads to
\begin{align*}
\frac{1}{4} \abs{y_{1} -1 + 1 - y_{2}}^{2} & \leq \frac{1}{2} ((y_{1} - 1)^{2} + (1 - y_{2})^{2})  = \int_{y_{2}}^{y_{1}} \abs{\abs{s} - 1} \ds.
\end{align*}
\item[(4)] $(y_{2} < y_{1} < 0)$, consider a similar argument to cases (1)-(3) with the subcases (4i) $y_{2} < y_{1} \leq -1$ where $\abs{\abs{s}-1} = -(1+s)$, (4ii) $-1 \leq y_{2} < y_{1} < 0$ where $\abs{\abs{s}-1} = 1+s$, and (4iii) $y_{2} < -1 < y_{1} < 0$ where we compute
\begin{align*}
\int_{y_{2}}^{-1} -(1+s) \ds + \int_{-1}^{y_{1}} 1+s \ds = \frac{1}{2} ((y_{1} - (-1))^{2} + (-1 - y_{2})^{2}) \geq \frac{1}{4} \abs{y_{1} - y_{2}}^{2}.
\end{align*}  
\item[(5)] $(y_{2} < 0 < y_{1})$, consider the four subcases (5i) $-1 < y_{2} < 0 < y_{1} < 1$, (5ii) $y_{2} < -1$, $y_{1} > 1$, (5iii) $y_{2} < -1$, $0 < y_{1} < 1$, and (5iv) $-1 < y_{2} < 0$, $y_{1} > 1$.  A short computation yields
\begin{align*}
2 \int_{y_{2}}^{y_{1}} \abs{\abs{s}-1} \ds = \begin{cases}
2 (y_{1} - y_{2}) - y_{1}^{2} - y_{2}^{2} & \text{ for } \mathrm{(5i)}, \\
2 + (y_{1}-1)^{2} + (1+y_{2})^{2} & \text{ for } \mathrm{(5ii)}, \\
2 - (1-y_{1})^{2} + (1+y_{2})^{2} & \text{ for } \mathrm{(5iii)}, \\
2 + (y_{1}-1)^{2} - (1+y_{2})^{2} & \text{ for } \mathrm{(5iv)}.
\end{cases}
\end{align*}
For (5i) observe that $y_{1} - y_{2} < 2$, and so $\abs{y_{1} - y_{2}}^{2} < 2 (y_{1} - y_{2})$.  Using also that $y_{1}^{2} \leq y_{1}$, $-y_{2} \geq y_{2}^{2}$ we infer 
\begin{align*}
2 (y_{1} - y_{2}) - y_{1}^{2} - y_{2}^{2} \geq y_{1} - y_{2} > \frac{1}{2} \abs{y_{1} - y_{2}}^{2}.
\end{align*}
By Young's inequality we have the inequalities $(y_{1} - y_{2})^{2} \leq 2 (y_{1}^{2} + y_{2}^{2})$, $-2 y_{1} \geq -\frac{1}{2}y_{1}^{2} - 2$ and $2 y_{2} \geq -\frac{1}{2}y_{2}^{2} - 2$, and so for (5ii)
\begin{align*}
2 + (y_{1}-1)^{2} + (1+y_{2})^{2} \geq \frac{1}{2} (y_{1}^{2} + y_{2}^{2}) \geq \frac{1}{4} \abs{y_{1}-y_{2}}^{2}.
\end{align*}
For (5iv), since $-1 < y_{2} < 0$, it holds that $-(1+y_{2})^{2} \geq -1$ and $y_{2}^{2} \leq 1$.  Thus, for $\delta \in (\frac{1}{2},1)$, $\eta \in (0, 2 - \delta^{-1})$, by Young's inequality $-2y_{1} \geq -\delta y_{1}^{2} - \delta^{-1}$ and adding a non-positive term $\eta(y_{2}^{2}-1)$, we infer that
\begin{align*}
2 + (y_{1}-1)^{2} - (1+y_{2})^{2} & \geq 2 + y_{1}^{2} - 2 y_{1}  \geq (1 - \delta) y_{1}^{2} + \eta y_{2}^{2} + 2 - \delta^{-1} - \eta \\
& \geq  \frac{1}{2} \min((1-\delta),\eta) \abs{y_{1}-y_{2}}^{2}.
\end{align*}
The last case (5iii) proceeds in a similar fashion by exchanging the roles of $y_{1}$ and $y_{2}$.
\end{enumerate}
The injectivity of $W$ follows from the first inequality of \eqref{Wfunctiondiff}.  From definition, $\tilde{\Psi}(y) = 0$ if and only if $y = \pm 1$, and the derivative $W'(y) = (\tilde{\Psi}(y))^{\frac{1}{2}}$ is positive for all $y \neq \pm 1$.  Hence $W(\cdot)$ is strictly increasing over $\R$, which together with injectivity implies that $W(\cdot)$ is bijective.
\end{proof}

\begin{lem}\label{AubinLionHolder}
Let $B$ be a Banach space and let $\{f_{n}\}_{n \in \N}$ be a bounded set in $L^{\infty}(0,T;B)$ such that, for every $0 \leq t_{1} < t_{2} \leq T$, the set $\{ \int_{t_{1}}^{t_{2}} f_{n}(t) \dt : n \in \N \}$ is relatively compact in $B$.  Moreover, suppose that there exists $\alpha \in (0,1)$ and a constant $C > 0$ such that for all $0 \leq \tau < t \leq T$ and $n \in \N$,
\begin{align}\label{uniformHolder}
\frac{\norm{f_{n}(t) - f_{n}(\tau)}_{B}}{\abs{t-\tau}^{\alpha}} \leq C.
\end{align}
Then $\{f_{n}\}_{n \in \N}$ is relatively compact in $C^{0,\beta}([0,T];B)$ for all $0 < \beta < \alpha$.
\end{lem}

\begin{proof}
Fix $h > 0$ and set $( \tau_{h}f)(t) = f(t+h)$ as the time translations of $f$.  Then, for $0 \leq t \leq T-h$, we have by \eqref{uniformHolder},
\begin{align}\label{uniformconv}
\sup_{t \in [0,T-h]} \norm{(\tau_{h}f_{n})(t) - f_{n}(t)}_{B} \leq \sup_{t \in [0,T-h]} C \abs{h}^{\alpha} = C \abs{h}^{\alpha} \to 0
\end{align}
as $h \to 0$ uniformly in $n \in \N$.  Hence, $\{f_{n}\}_{n \in \N}$ is relatively compact in $C^{0}([0,T];B)$ \cite[Thm.~1]{Simon86}, and so there exists a function $f \in C^{0}([0,T];B)$ such that along a subsequence $f_{n_{j}} \to f$ strongly in $C^{0}([0,T];B)$ as $j \to \infty$.  Alternatively, one may use \eqref{uniformHolder} to deduce equicontinuity of $\{f_{n}\}_{n \in \N}$ and apply the Banach-space-valued Arzel\`{a}--Ascoli theorem \cite[Ch.~III, \S~3, Thm.~3.1]{Lang}.  Furthermore, passing to the limit $n_{j} \to \infty$ in \eqref{uniformHolder} yields
\begin{align*}
\sup_{0 \leq \tau < t \leq T} \left ( [f]_{\alpha, B}(\tau, t) \right ) \leq C,
\end{align*}
and so the limit function $f$ belongs to $C^{0,\alpha}([0,T];B)$.  It remains to show that 
\begin{align}\label{Holdernorm:conv}
\sup_{0 \leq \tau < t \leq T} \left ( [f_{n_{j}} - f]_{\beta, B}(\tau, t) \right ) \to 0 \text{ as } j \to \infty
\end{align}
for any $\beta \in (0,\alpha)$, and the proof is similar to the proof of compact embeddings in H\"{o}lder spaces.  Let $\eta > 0$ and $\delta > 0$ be arbitrary and for $0 \leq \tau < t \leq T$, set $X_{1} := \{ (\tau, t) : \abs{t - \tau} \leq \delta\}$ and $X_{2} := \{ (\tau,t) : \abs{t - \tau} \geq \delta \}$.  Then, a short computation shows that
\begin{align*}
\mathrm{(1)} & \quad \sup_{(\tau,t) \in X_{1}} \left ( [f_{n_{j}} - f]_{\beta, B}(\tau,t) \right ) = \sup_{(\tau,t) \in X_{1}} \left ( \abs{t-\tau}^{\alpha - \beta} [f_{n_{j}} - f]_{\alpha, B}(\tau,t) \right ) \\
& \quad \quad  \leq \sup_{(\tau,t) \in X_{1}} \left ( [f_{n_{j}}]_{\alpha, B}(\tau,t) + [f]_{\alpha,B}(\tau, t) \right ) \delta^{\alpha - \beta} \leq C \delta^{\alpha - \beta}, \\
\mathrm{(2)} & \quad \sup_{(\tau,t) \in X_{2}} \left ([f_{n_{j}} - f]_{\beta, B}(\tau,t) \right ) \leq 2 \delta^{-\beta} \norm{f_{n_{j}} - f}_{C^{0}([0,T];B)},
\end{align*}
and so upon choosing $\delta$ sufficiently small such that $C \delta^{\alpha - \beta} < \frac{1}{2} \eta$, and then choosing $j$ sufficiently large so that $C \delta^{-\beta} \norm{f_{n_{j}} - f}_{C^{0}([0,T];B)} < \frac{1}{2} \eta$, we arrive at
\begin{align*}
\sup_{0 \leq \tau < t \leq T} \left ( [f_{n_{j}} - f]_{\beta, B}(\tau, t) \right ) \leq C \left ( \delta^{\alpha-\beta} +  \delta^{-\beta} \norm{f_{n_{j}} - f}_{C^{0}([0,T];B)} \right ) < \eta,
\end{align*}
for arbitrary $\eta > 0$.  This yields \eqref{Holdernorm:conv}.
\end{proof}

\section{Main result}\label{sec:Main}
We introduce a new variable $\e{\theta}$ defined as
\begin{align}\label{theta:defn}
\e{\theta} := \e{\sigma} - \chi \e{\varphi}, 
\end{align}
and consider the Cahn--Hilliard--Darcy system
\begin{subequations}\label{CHD}
\begin{alignat}{3}
\div \e{\bm{v}} & = \HH && \text{ in } Q, \label{CHD:div} \\
\K \e{\bm{v}} & = - \nabla \e{p} + (\e{\mu} + \chi \e{\theta} + \chi^{2} \e{\varphi}) \nabla \e{\varphi} && \text{ in } Q, \label{CHD:Darcy} \\
\pd_{t} \e{\varphi} + \div (\e{\varphi} \e{\bm{v}}) & = \div (m(\e{\varphi}) \nabla \e{\mu}) + \U \e{\varphi} && \text{ in } Q, \label{CHD:varphi} \\
\e{\mu} & = \eps^{-1} \Psi'(\e{\varphi}) - \eps \Laplace \e{\varphi} - \chi \e{\theta} - \chi^{2} \e{\varphi} && \text{ in } Q, \label{CHD:mu} \\
\pd_{t} (\e{\theta} + \chi \e{\varphi}) + \div ( (\e{\theta} + \chi \e{\varphi}) \e{\bm{v}}) & = \div (n(\e{\varphi}) \nabla \e{\theta}) + \So && \text{ in } Q, \label{CHD:theta} \\
\e{\bm{v}} \cdot \bm{\nu} = \pdnu \e{\varphi} & = \pdnu \e{\mu} = \pdnu \e{\theta} = 0 && \text{ on } \Sigma, \label{CHD:bc} \\
\e{\varphi}(0) = \e{\varphi}_{0}, \quad \e{\theta}(0) & = \e{\sigma}_{0} - \chi \e{\varphi}_{0} && \text{ in } \Omega. \label{CHD:ini}
\end{alignat}
\end{subequations}
The energy associated to \eqref{CHD} is
\begin{align}\label{CHD:energy}
\EE(\e{\varphi}, \e{\theta}) = \int_{\Omega} \frac{1}{\eps} \Psi(\e{\varphi}) + \frac{\eps}{2} \abs{\nabla \e{\varphi}}^{2} + \frac{1}{2} \abs{\e{\theta}}^{2} - \frac{\chi^{2}}{2} \abs{\e{\varphi}}^{2} \dx,
\end{align}
and solutions satisfy the energy identity
\begin{equation}\label{CHD:energy:id}
\begin{aligned}
& \EE(\e{\varphi}(t), \e{\theta}(t)) + \int_{Q_{t}} m(\e{\varphi}) \abs{\nabla \e{\mu}}^{2} + n(\e{\varphi}) \abs{\nabla \e{\theta}}^{2} + \K \abs{\e{\bm{v}}}^{2} \dx \ds \\
& \quad = \EE(\e{\varphi}_{0}, \e{\theta}_{0}) + \int_{Q_{t}} \U \e{\varphi} \e{\mu} + \So \e{\theta} + \HH \left( \e{p} - \e{\varphi} \e{\mu} - \tfrac{1}{2} \abs{\e{\theta} + \chi \e{\varphi}}^{2} \right ) \dx \ds,
\end{aligned}
\end{equation}
which can be obtained by testing \eqref{CHD:varphi} with $\e{\mu}$, \eqref{CHD:mu} with $\pd_{t} \e{\varphi}$, \eqref{CHD:theta} with $\e{\theta}$, \eqref{CHD:Darcy} with $\e{\bm{v}}$, summing the resulting equalities, integrating by parts and using \eqref{CHD:div}, and integrating in time.  By taking the divergence of \eqref{CHD:Darcy} and using \eqref{CHD:div}, we obtain an elliptic equation for the pressure $\e{p}$:
\begin{equation}\label{Pressure:poisson}
\begin{aligned}
-\Laplace \e{p} = \K \HH - \div ((\e{\mu} + \chi \e{\theta} + \chi^{2} \e{\varphi}) \nabla \e{\varphi}) & \text{ in } Q, \\
\pdnu \e{p} = 0 & \text{ on } \Sigma,
\end{aligned}
\end{equation}
with homogeneous Neumann boundary conditions.  Since the addition of a time-dependent constant is also a solution to the above elliptic equation, we demand $\e{p}$ satisfies the mean-zero condition $\mean{\e{p}} = 0$ for unique solvability.  Furthermore, integrating \eqref{CHD:div} and the boundary condition $\e{\bm{v}} \cdot \bm{\nu} = 0$ on $\Sigma$ necessary implies that the source term $\HH$ has zero spatial mean.

The formal sharp interface limit of \eqref{CHD} is
\begin{subequations}\label{CHD:SIM}
\begin{alignat}{3}
\div \bm{v} = \HH , \quad \K \bm{v} = - \nabla p  & \text{ in } \Omega_{+} \cup \Omega_{-}, \label{S:div} \\
- \div (m \nabla \mu) = \U - \HH, \quad \pd_{t} \theta + \div ( \theta \bm{v} - n \nabla \theta) = \So - \chi \HH & \text{ in } \Omega_{+}, \label{S:+} \\
- \div (m \nabla \mu) = - \U + \HH, \quad \pd_{t} \theta + \div ( \theta \bm{v} - n \nabla \theta) = \So + \chi \HH & \text{ in } \Omega_{-}, \label{S:-} \\
\jump{\bm{v}} \cdot \bm{\nu} = 0, \quad \jump{\mu} = 0, \quad \jump{\theta} = 0, \quad 2 \mu + 2 \chi \theta  = \jump{p} = \kappa & \text{ on } \Gamma,  \label{S:G1} \\
2(-\velo + \bm{v} \cdot \bm{\nu}) = \jump{m \nabla \mu} \cdot \bm{\nu}, \quad 2 \chi (-\velo + \bm{v} \cdot \bm{\nu}) = \jump{n \nabla \theta} \cdot \bm{\nu} & \text{ on } \Gamma, \label{S:G2} \\
m \pdnu \mu = 0, \quad n \pdnu \theta = 0, \quad \bm{v} \cdot \bm{\nu} = 0 & \text{ on } \pd \Omega \setminus \Gamma. \label{S:G3}
\end{alignat}
\end{subequations}
In the above $\kappa, \mathcal{V}$ and $\bm{\nu}$ denote the mean curvature, normal velocity, and unit normal of the interface $\Gamma$, respectively, and we have used the notation 
\begin{align*}
m = \begin{cases}
m(1) & \text{ in } \Omega_{+}, \\ m(-1) & \text{ in } \Omega_{-},
\end{cases} \quad n = \begin{cases}
n(1) & \text{ in } \Omega_{+}, \\ n(-1) & \text{ in } \Omega_{-}.
\end{cases}
\end{align*}

We remark that the jump condition $\jump{\theta} = 0$ is equivalent to the jump condition $\jump{\sigma} = 2 \chi$ as observed in the formal analysis in Garcke et al.~\cite{GLSS}.  Indeed, we will show below that $\e{\varphi}$ converges (along a subsequence) to $-1 + 2 \chara_{\Omega_{+}}$ for some set $\Omega_{+} \subset Q$ of bounded variation.  If $\theta$ and $\sigma$ are the limits of $\{\e{\theta}\}_{\eps > 0}$ and $\{\e{\sigma}\}_{\eps > 0}$, respectively, then by \eqref{theta:defn} one obtains
\begin{align*}
\theta = \sigma - \chi + 2 \chi \chara_{\Omega_{+}} = \begin{cases}
\sigma_{+} + \chi & \text{ in } \Omega_{+}, \\
\sigma_{-} - \chi & \text{ in } \Omega_{-},
\end{cases} \quad \text{ so that } \jump{\sigma} = 2 \chi \; \Leftrightarrow \, \jump{\theta} = 0.
\end{align*}
Hence, if we can show that in the limit $\{\e{\theta}\}_{\eps > 0}$ converges to a function $\theta$ with $\jump{\theta} = 0$ on the interface $\Gamma$, we have ascertained the jump condition \eqref{Intro:jump} in the sharp interface limit when the parameter $\chi$ is non-zero.

The main result on the sharp interface limit of \eqref{CHD} is formulated as follows.
\begin{thm}[Sharp interface limit]\label{thm:main}
For $\eps > 0$, let $(\e{\varphi}, \e{\mu},\e{\theta}, \e{\bm{v}}, \e{p})$ be a solution to \eqref{CHD} with initial data $(\e{\varphi}_{0}, \e{\sigma}_{0})$.  Then, there exists a sequence $\{\eps_{j}\}_{j \in \N}$, $\eps_{j} \to 0$ as $j \to \infty$ with the following properties.

\begin{enumerate}
\item[$\mathrm{(1)}$] There exists a measurable set $\Omega_{+} \subset Q$ with
\begin{align*}
\Omega_{+} := \bigcup_{t \in [0,T]} \Omega_{+}(t) \times \{t\}, \quad \Omega_{+}(t) \subset \Omega \quad \text{ for all } t \in [0,T]
\end{align*}
such that $\chara_{\Omega_{+}} \in L^{\infty}(0,T;BV(\Omega, \{0,1\})) \cap C^{0,\frac{1}{4}}([0,T];L^{1}(\Omega))$, $\chara_{\Omega_{+}} \vert_{t=0} = \chara_{\Omega_{+}(0)}$ in $L^{2}(\Omega)$, and
\begin{align*}
\varphi^{\eps_{j}} \to \varphi := -1 + 2 \chara_{\Omega_{+}(t)} \text{ a.e.~in } Q \text{ and strongly in } C^{0,\alpha}([0,T];L^{2}(\Omega))
\end{align*}
for any $\alpha \in (0, \frac{1}{8})$.

\item[$\mathrm{(2)}$] There exist Radon measures $\lambda, \{\lambda_{ik}\}_{1 \leq i, k \leq d}$ on $Q$, Radon measures $\lambda^{t}$, $\{\lambda^{t}_{ik}\}_{1 \leq i,k \leq d}$ on $\Omega$ for a.e.~$t \in (0,T)$, such that $\dd \lambda(x,t) = \dd \lambda^{t}(x) \dt$, $\dd \lambda_{ik}(x,t) = \dd \lambda^{t}_{ik}(x) \dt$ and
\begin{align*}
\left ( \frac{1}{\eps_{j}} \Psi(\varphi^{\eps_{j}}) + \frac{\eps_{j}}{2} \abs{\nabla \varphi^{\eps_{j}}}^{2} \right ) \dx \dt & \to \lambda \text{ weakly* in } \M(Q), \\
\left ( \eps_{j} \pd_{x_{i}} \varphi^{\eps_{j}} \pd_{x_{k}} \varphi^{\eps_{j}} \right ) \dx \dt & \to \lambda_{ik} \text{ weakly* in } \M(Q),
\end{align*}
with
\begin{align*}
\abs{\nabla \chara_{\Omega_{+}(t)}}(\Omega) \dx & \leq \dd \lambda^{t}(x) \text{ for a.e. } t \in (0,T).
\end{align*}

\item[$\mathrm{(3)}$] There exists a Radon measure $V = V^{t} \dt$ on $Q \times \PP^{d-1}$ such that
\begin{align*}
\int_{0}^{T} \inner{\delta V^{t}}{\bm{Y}} \dt = \int_{Q} \nabla \bm{Y} : \left ( \id \dd \lambda^{t} - (\dd \lambda^{t}_{ik} \, )_{ik} \right ) \dt \quad \forall \bm{Y} \in C^{1}_{0}(\Omega; \R^{d}).
\end{align*}
Furthermore, for a.e.~$t \in (0,T)$, there exist $\lambda^{t}$-measurable functions $\{c^{t}_{i}\}_{1 \leq i \leq d}$ and vector-valued functions $\{\bm{\nu}^{t}_{i}\}_{1 \leq i \leq d}$ such that
\begin{align*}
0 \leq c^{t}_{i} \leq 1, \quad \sum_{i=1}^{d} c^{t}_{i} \geq 1, \quad \sum_{i=1}^{d} \bm{\nu}^{t}_{i} \otimes \bm{\nu}^{t}_{i} = \id \quad \lambda^{t}-\text{a.e.}, 
\end{align*}
and the varifold $V = V^{t} \dt$ has the following representation
\begin{align}\label{Vari:represent}
\int_{\Omega \times \PP^{d-1}} Y(x,P) \dd V^{t}(x,P) = \int_{\Omega} \sum_{i=1}^{d} c^{t}_{i}(x) Y(x, \bm{\nu}^{t}_{i}(x)) \dd \lambda^{t}(x),
\end{align}
for all $Y \in C^{0}_{c}(\Omega \times \PP^{d-1})$.

\item[$\mathrm{(4)}$] There exist functions $\mu \in L^{2}(0,T;H^{1}(\Omega))$, $\theta \in C^{0}_{\mathrm{w}}([0,T];L^{2}(\Omega)) \cap L^{2}(0,T;H^{1}(\Omega))$, $\bm{v} \in L^{2}(Q)^{d}$ and $p \in L^{2}(0,T;L^{2}_{0}(\Omega))$ such that
\begin{subequations}
\begin{alignat}{3}
\mu^{\eps_{j}} \; (\text{resp. } \theta^{\eps_{j}}) & \longrightarrow \mu \; (\text{resp. } \theta) \text{ weakly in } L^{2}(0,T;H^{1}(\Omega)), \label{mu:theta:L2H1} \\
\theta^{\eps_{j}} & \longrightarrow \theta \text{ weakly* in } L^{\infty}(0,T;L^{2}(\Omega)), \label{theta:LinftyL2}\\
\theta^{\eps_{j}} & \longrightarrow \theta \text{ strongly in } L^{2}(Q), \label{theta:L2L2} \\
\bm{v}^{\eps_{j}} & \longrightarrow \bm{v} \text{ weakly in } L^{2}(Q)^{d}, \label{v:L2L2} \\
p^{\eps_{j}} & \longrightarrow p \text{ weakly in } L^{2}(Q), \label{p:L2L2} 
\end{alignat}
\end{subequations}
with 
\begin{align}\label{theta:ini:attain}
\inner{\theta(0)}{\phi} = \inner{\theta_{0}}{\phi} = \inner{\sigma_{0} + \chi - 2 \chi \chara_{\Omega_{+}(0)}}{\phi} \text{ for all } \phi \in L^{2}(\Omega).
\end{align}

\item[$\mathrm{(5)}$] Let $\varphi := -1 + 2 \chara_{\Omega_{+}}$.  Then, the quintuple $(V, \mu, \theta, \bm{v}, p)$ is a varifold solution to the sharp interface limit \eqref{CHD:SIM} with initial values $(\chara_{\Omega_{+}(0)}, \theta_{0})$ in the sense that
\begin{subequations}\label{CHD:SIM:varifold}
\begin{alignat}{3}
\label{CHD:SIM:div}
0 & = \int_{Q} \bm{v} \cdot \nabla \zeta + \zeta \HH \dx \dt, \\
\label{CHD:SIM:Darcy}
0 & = \int_{Q} \K \bm{v} \cdot \bm{Y} - p \div \bm{Y} + 2 \chara_{\Omega_{+}} \div (( \mu + \chi \theta) \bm{Y}) \dx \dt, \\
\label{CHD:SIM:mu} 
0 & = \int_{Q} m(\varphi)  \nabla \mu \cdot \nabla \zeta  - 2 \chara_{\Omega_{+}} (\pd_{t} \zeta + \nabla \zeta \cdot \bm{v} + \U \zeta) + (\U - \HH) \zeta \dx \dt \\
\notag & \quad - \int_{\Omega} 2 \chara_{\Omega_{+}(0)} \zeta(\cdot, 0) \dx, \\
\label{CHD:SIM:theta} 
0 & = \int_{Q} n(\varphi) \nabla \theta \cdot \nabla \zeta - (2 \chi \chara_{\Omega_{+}(t)} + \theta) (\pd_{t}\zeta + \nabla \zeta \cdot \bm{v}) \dx \dt \\
\notag & \quad - \int_{Q}  \So \zeta + \chi \HH \zeta \dx \dt - \int_{\Omega} (2 \chi \chara_{\Omega_{+}(0)} + \theta_{0}) \zeta(\cdot,0) \dx, \\
\label{CHD:SIM:kappa} 
0 & = \int_{Q} 2 \chara_{\Omega_{+}} \div ((\mu + \chi \theta) \bm{Y}) \dx \dt - \int_{0}^{T} \inner{\delta V^{t}}{\bm{Y}} \dt
\end{alignat}
\end{subequations}
hold for all $\zeta \in C^{\infty}(\overline{Q})$ such that $\zeta(T) = 0$ and for all $\bm{Y} \in C^{0}([0,T];C^{1}_{0}(\Omega; \R^{d}))$.  Furthermore, for a.e.~$0 \leq \tau < t \leq T$, it holds that 
\begin{equation}\label{CH:SIM:energy}
\begin{aligned}
& \lambda^{t}(\Omega) + \frac{1}{2} \norm{\theta(t)}_{L^{2}}^{2} + \int_{\tau}^{t} \int_{\Omega} m(\varphi) \abs{\nabla \mu}^{2} + n(\varphi) \abs{\nabla \theta}^{2} + \K \abs{\bm{v}}^{2} \dx \ds \\
& \leq \lambda^{\tau}(\Omega) + \frac{1}{2} \norm{\theta(\tau)}_{L^{2}}^{2} + \int_{\tau}^{t} \int_{\Omega} \U \varphi \mu + \So \theta + \HH \left ( p - \mu \varphi + \tfrac{1}{2} \abs{\theta + \chi \varphi}^{2} \right ) \dx \ds.
\end{aligned}
\end{equation}
\end{enumerate}
\end{thm}
\begin{remark}
The jump conditions $\jump{\mu} = \jump{\theta} = 0$ in \eqref{CHD:SIM} are implicitly encoded in the requirement that $\mu$ and $\theta$ belong to $L^{2}(0,T;H^{1}(\Omega))$.  
\end{remark}

\begin{remark}
Theorem~\ref{thm:main} generalizes the result of Melchionna and Rocca \cite{RM} and of Fei \cite{Fei} (case $\alpha = 0$) by considering source terms $\U \varphi$ and $\HH$, and a coupling to a convection-reaction-diffusion equation through the non-zero parameter $\chi$.
\end{remark}
We end this section with a brief illustration on how \eqref{CHD:SIM:varifold} forms an appropriate weak formulation for \eqref{CHD:SIM}.  Assume for the moment that \eqref{CHD:SIM} has a sufficiently smooth solution $(\Omega_{\pm}, \Gamma, \mu, \theta, \bm{v}, p, \kappa)$, and denote the integration with respect to the Hausdorff measure on $\Gamma$ by $\dhaus$.  We apply Reynold's transport theorem to $\int_{\Omega_{+}(t)} \zeta \dx$ for functions $\zeta \in C^{1}(\overline{Q})$ such that $\zeta(T) = 0$, leading to
\begin{align}\label{Reynolds}
-\int_{\Omega} \chara_{\Omega_{+}(0)} \zeta(\cdot, 0) \dx = \int_{0}^{T} \int_{\Omega} \chara_{\Omega_{+}(t)} \pd_{t}\zeta \dx \dt - \int_{0}^{T} \int_{\Gamma(t)} \velo \zeta \dhaus \dt,
\end{align}
where we recall that $\velo$ is the normal velocity of $\Gamma$ associated to $\bm{\nu}$, the unit normal pointing into $\Omega_{+}$, and hence the appearance of a minus sign on the last term of the right-hand side.  Employing $\div \bm{v} = \HH$ and the divergence theorem on $\Omega_{+}(t)$ we find that
\begin{equation}\label{Reynolds2}
\begin{aligned}
\int_{0}^{T} \int_{\Gamma(t)} 2 (-\velo + \bm{v} \cdot \bm{\nu}) \zeta \dhaus \dt & = \int_{0}^{T} \int_{\Omega} -2 \chara_{\Omega_{+}(t)} (\pd_{t} \zeta + \bm{v} \cdot \nabla \zeta + \HH \zeta) \dx \dt \\
& \quad - \int_{\Omega} 2 \chara_{\Omega_{+}(0)} \zeta(\cdot, 0) \dx.
\end{aligned}
\end{equation}
Then, the derivation of \eqref{CHD:SIM:mu} and \eqref{CHD:SIM:theta} follows directly from \eqref{Reynolds2} and the equations involving $\mu$ and $\theta$ in \eqref{CHD:SIM}.  We point out that one has to use $\jump{\bm{v}} \cdot \bm{\nu} = 0$ in the derivation of \eqref{CHD:SIM:theta}.  The derivation of \eqref{CHD:SIM:div} is clear from testing $\div \bm{v} = \HH$ in $\Omega_{\pm}$ with an arbitrary test function $\zeta \in C^{1}(\overline{Q})$ and using $\jump{\bm{v}} \cdot \bm{\nu} = 0$ on $\Gamma$.  Meanwhile, using the well-known result that the first variation of the area of $\Gamma$ in the direction of an arbitrary vector field $\bm{Y} \in C^{1}(\Omega; \R^{d})$ is 
\begin{align*}
\frac{\delta}{\delta \Gamma} \mathrm{Area}(\Gamma)[\bm{Y}] = \int_{\Gamma} \div_{\Gamma} \bm{Y} \dhaus = \int_{\Gamma} - \kappa \bm{Y} \cdot \bm{\nu} \dhaus,
\end{align*}
where $\div_{\Gamma}$ denotes the surface divergence and the above surface integration-by-parts formula differs to the classical formulation by a sign since $\bm{\nu}$ points into $\Omega_{+}$, we obtain for the relation $2 (\mu + \chi \theta) = \kappa$ the weak formulation
\begin{align*}
\int_{\Omega_{+}(t)} 2 \div \left ( \left ( \mu + \chi \theta \right ) \bm{Y} \right ) \dx & = -\int_{\Gamma(t)} 2 (\mu + \chi \theta) \bm{\nu} \cdot \bm{Y} \dhaus \\
& = - \int_{\Gamma(t)} \kappa \bm{\nu} \cdot \bm{Y} \dhaus = \frac{\delta}{\delta \Gamma} \mathrm{Area}(\Gamma)[\bm{Y}].
\end{align*}
This motivates \eqref{CHD:SIM:kappa}.  Lastly testing the Darcy law $\K \bm{v} = - \nabla p$ in $\Omega_{\pm}$ with an arbitrary $\bm{Y} \in C^{1}(\overline{\Omega};\R^{d})$ yields
\begin{align*}
\int_{\Omega} \K \bm{v} \cdot \bm{Y} - p \div \bm{Y} \dx & = -\int_{\Gamma} \jump{p}\bm{\nu} \cdot \bm{Y} \dhaus = -\int_{\Gamma} 2(\mu + \chi \theta) \bm{\nu} \cdot \bm{Y} \dhaus \\
& = - \int_{\Omega} 2 \chara_{\Omega_{+}} \div \left ( \left ( \mu + \chi \theta \right ) \bm{Y} \right ) \dx,
\end{align*}
which motivates \eqref{CHD:SIM:Darcy}.

\section{Uniform estimates}\label{sec:Est}
In this section, the symbol $C$ will denote constants that are independent of $\eps$ and may vary from line to line.  
\subsection{First estimate}\label{sec:FirstEst}
Let 
\begin{align}\label{defn:eps0}
\eps_{0} := \min \left (1, k_{0} / \chi^{2} \right ),
\end{align}
where $k_{0}$ is the constant in \eqref{Psi:lowerbound}.  By \eqref{Psi:lowerbound} it holds that
\begin{equation}\label{varphiL2:Psi:L1}
\begin{aligned}
\norm{\varphi}_{L^{2}}^{2} & \leq \int_{\{\abs{\varphi} \leq 1-c_{0}\}} (1-c_{0})^{2} \dx + \int_{\{\abs{\varphi} > 1-c_{0}\}} k_{0}^{-1} \Psi(\varphi) + k_{1} \dx  \\
& \leq k_{0}^{-1} \norm{\Psi(\varphi)}_{L^{1}} + C.
\end{aligned}
\end{equation}
For $\eps \in (0,\eps_{0}]$, we have $\frac{1}{2 \eps} \geq \frac{\chi^{2}}{2k_{0}}$, and so together with Young's inequality, \eqref{varphiL2:Psi:L1} yields a lower bound for the energy $\mathcal{E}$ defined in \eqref{CHD:energy}:
\begin{equation}\label{CH:energy:density:LB}
\begin{aligned}
\EE(\e{\varphi}(t), \e{\theta}(t)) & \geq \int_{\Omega} \left ( \frac{1}{\eps} - \frac{\chi^{2}}{2 k_{0}} \right ) \Psi(\e{\varphi}(t)) + \frac{\eps}{2} \abs{\nabla \e{\varphi}(t)}^{2} + \frac{1}{2} \abs{\e{\theta}(t)}^{2} \dx - C \\
& \geq \int_{\Omega} \frac{1}{2 \eps} \Psi(\e{\varphi}(t)) + \frac{\eps}{2} \abs{\nabla \e{\varphi}(t)}^{2} + \frac{1}{2}\abs{\e{\theta}(t)}^{2} \dx - C \quad \forall \eps \in (0,\eps_{0}].
\end{aligned}
\end{equation}
We now consider the energy identity \eqref{CHD:energy:id} and estimate the right-hand side as follows: By \eqref{ass:source} and the Poincar\'{e} inequality, we see that
\begin{equation}\label{RHS:ZYmu:1}
\begin{aligned}
\int_{\Omega} \U \e{\varphi} \e{\mu} + \So \e{\theta} \dx & = \int_{\Omega} \U \e{\varphi} (\e{\mu} - \mean{\e{\mu}}) + \So \e{\theta} \dx + \mean{\e{\mu}} \int_{\Omega} \U \e{\varphi} \dx \\
& \leq C \norm{\U}_{L^{\infty}} \norm{\e{\varphi}}_{L^{2}} \norm{\nabla \e{\mu}}_{L^{2}} + \norm{\So}_{L^{2}} \norm{\e{\theta}}_{L^{2}} + \mean{\e{\mu}} \int_{\Omega} \U \e{\varphi} \dx .
\end{aligned}
\end{equation}
To handle the last term involving the mean $\mean{\e{\mu}}$ we adapt the proof of \cite[Lem. 3.4]{Chen96} as follows:  Let $\psi$ be a $C^{2}(\Omega)$-function with $\pdnu \psi = 0$ on $\pd \Omega$, and denote
\begin{align}\label{e:GL}
\e{e} := \frac{1}{\eps} \Psi(\e{\varphi}) + \frac{\eps}{2} \abs{\nabla \e{\varphi}}^{2}, \quad \e{A} := \e{e} - \eps \nabla \e{\varphi} \otimes \nabla \e{\varphi}.
\end{align}
Then, we obtain after multiplying \eqref{CHD:mu} with $\nabla \psi \cdot \nabla \e{\varphi}$, integrating over $\Omega$ and integrating by parts 
\begin{align*}
& \int_{\Omega} \e{\mu} \nabla \psi \cdot \nabla \e{\varphi} \dx = \int_{\Omega} \left ( \frac{1}{\eps} \Psi'(\e{\varphi}) -\eps \Laplace \e{\varphi} - \chi \e{\theta} - \chi^{2} \e{\varphi} \right ) \nabla \psi \cdot \nabla \e{\varphi} \dx \\
& \;  = \int_{\Omega} \nabla \e{e} \cdot \nabla \psi - \eps \div (\nabla \e{\varphi} \otimes \nabla \e{\varphi}) \cdot \nabla \psi - \chi (\e{\theta} + \chi \e{\varphi}) \nabla \psi \cdot \nabla \e{\varphi} \dx \\
& \; = \int_{\Omega} - \e{A} \id : D^{2} \psi + \chi \left ( \nabla \e{\theta} \cdot \nabla \psi + \left ( \e{\theta} + \frac{\chi}{2} \e{\varphi} \right ) \Laplace \psi \right ) \e{\varphi} \dx, 
\end{align*}
where $\id$ denotes the identity tensor, $D^{2}\psi$ denotes the Hessian of $\psi$ and we used the relation
\begin{align*}
\frac{1}{2} \nabla \left ( \abs{\nabla f}^{2} \right ) = \div (\nabla f \otimes \nabla f) - \Laplace f \; \nabla f.
\end{align*}
Meanwhile, on the left-hand side we have
\begin{align*}
\int_{\Omega} \e{\mu} \nabla \psi \cdot \nabla \e{\varphi} \dx = -\int_{\Omega} \e{\varphi} \nabla \e{\mu} \cdot \nabla \psi + (\e{\mu} - \mean{\e{\mu}}) \e{\varphi} \Laplace \psi \dx - \mean{\e{\mu}} \int_{\Omega} \e{\varphi} \Laplace \psi \dx,
\end{align*}
so that upon combining we obtain
\begin{equation}\label{RHS:mean:mu}
\begin{aligned}
\mean{\e{\mu}} \int_{\Omega} \e{\varphi} \Laplace \psi \dx & = \int_{\Omega} D^{2} \psi : \e{A} \id  - \tfrac{1}{2} \chi^{2}  \abs{\e{\varphi}}^{2} \Laplace \psi \dx \\
& \quad - \int_{\Omega} (\nabla (\e{\mu} + \chi \e{\theta}) \cdot \nabla \psi + (\e{\mu} - \mean{\e{\mu}} + \e{\theta}) \Laplace \psi ) \, \e{\varphi} \dx \\
& \leq C \left ( 1 + \norm{\psi}_{C^{2}(\Omega)}^{2} \right ) \left ( \int_{\Omega} \e{e} + \abs{\e{\varphi}}^{2} + \abs{\e{\theta}}^{2} \dx \right ) \\
& \quad + \frac{1}{8}\norm{m^{\frac{1}{2}}(\e{\varphi}) \nabla \e{\mu}}_{L^{2}}^{2} + \frac{1}{4}\norm{n^{\frac{1}{2}}(\e{\varphi}) \nabla \e{\theta}}_{L^{2}}^{2},
\end{aligned}
\end{equation}
where we have used the boundedness and positivity of the mobilities $m(\cdot)$ and $n(\cdot)$, i.e., $\norm{\nabla \e{\mu}}_{L^{2}} \leq \frac{1}{\sqrt{m_{0}}} \norm{m^{\frac{1}{2}}(\e{\varphi}) \nabla \e{\mu}}_{L^{2}}$.  Choosing $\psi$ to be the unique solution to
\begin{align*}
-\Laplace \psi = -\U \text{ in } \Omega, \quad \pdnu \psi = 0 \text{ on } \pd \Omega \quad  \text{ with } \mean{\psi} = 0,
\end{align*}
which is possible as $\U$ has zero spatial mean, then the left-hand side of \eqref{RHS:mean:mu} reads as $\mean{\e{\mu}} \int_{\Omega} \U \e{\varphi} \dx$, while by the assumption \eqref{ass:source} that $\U \in C^{0}(\overline{\Omega})$ and classical elliptic theory yields $\psi \in C^{2}(\Omega) \cap C^{0}(\overline{\Omega})$ satisfying $\norm{\psi}_{C^{2}(\Omega)} \leq C \norm{\U}_{C^{0}(\Omega)} \leq C$.  Thus, substituting the above estimate to the right-hand side of \eqref{RHS:ZYmu:1}, and integrating in time from $0$ to $t$, we arrive at
\begin{equation}\label{RHS:ZYmu:2}
\begin{aligned}
\int_{Q_{t}} \U \e{\varphi} \e{\mu} + \So \e{\theta} \dx \ds & \leq C \left ( \int_{Q_{t}} \e{e} + \abs{\e{\varphi}}^{2} + \abs{\e{\theta}}^{2} \dx \ds \right ) + C \norm{\So}_{L^{2}(Q)}^{2}\\
& \quad + \frac{2}{8}\norm{m^{\frac{1}{2}}(\e{\varphi}) \nabla \e{\mu}}_{L^{2}(Q_{t})}^{2} + \frac{1}{4} \norm{n^{\frac{1}{2}}(\e{\varphi}) \nabla \e{\theta}}_{L^{2}(Q_{t})}^{2}.
\end{aligned}
\end{equation}
To estimate the source term involving $\HH \in L^{\infty}(0,T;L^{\infty}_{0}(\Omega))$, we employ the splitting (cf. \cite{JWZ})
\begin{equation}\label{RHS:HH}
\begin{aligned}
& \int_{Q_{t}} \HH \left (\e{p} - \e{\mu} \e{\varphi} - \tfrac{1}{2} \abs{\e{\theta} + \chi \e{\varphi}}^{2} \right ) \dx \ds \\
& \quad  = \int_{Q_{t}} \HH (\e{p} - \overline{\e{\mu}} \e{\varphi}) \dx \dt + \int_{Q_{t}} \HH \left ( \e{\varphi} (\mean{\e{\mu}} - \e{\mu}) - \tfrac{1}{2} \abs{\e{\theta} + \chi \e{\varphi}}^{2} \right ) \dx \ds \\
& \quad \leq \int_{Q_{t}} \HH (\e{p} - \overline{\e{\mu}} \e{\varphi}) \dx \ds +  \frac{1}{8} \norm{m^{\frac{1}{2}}(\e{\varphi})\nabla \e{\mu}}_{L^{2}(Q_{t})}^{2} +  C \norm{\e{\theta}}_{L^{2}(Q_{t})}^{2} + C\norm{\e{\varphi}}_{L^{2}(Q_{t})}^{2},
\end{aligned}
\end{equation}
Introducing the operator $\NN: L^{2}_{0} \to H^2_n \cap L^{2}_{0}$ as $\NN(f) = u$ where $u$ is the unique solution to 
\begin{align*}
-\Laplace u = f \text{ in } \Omega, \quad \pdnu u = 0 \text{ on } \pd \Omega \quad \text{ with } \mean{u} = 0.
\end{align*}
By the Lax--Milgram theorem and elliptic regularity the solution $u$ lies in $H^2_n\cap L^{2}_{0}$ with $\norm{u}_{H^{2}} \leq C \norm{f}_{L^{2}}$.  Then, it is easy to check that $\NN(-\mean{\e{\mu}} \Laplace (\e{\varphi} - \mean{\e{\varphi}})) = \mean{\e{\mu}} (\e{\varphi} - \mean{\e{\varphi}})$ and so
\begin{align*}
 \e{p} &= \NN \left ( \K \HH - \div \left [ \left ( \e{\mu} - \mean{\e{\mu}} + \chi \e{\theta} + \chi^{2} \e{\varphi} \right ) \nabla \e{\varphi} \right ] - \mean{\e{\mu}} \Laplace (\e{\varphi} - \mean{\e{\varphi}}) \right ) \\
 & = \NN \left ( \K \HH - \div \left [ \left ( \e{\mu} - \mean{\e{\mu}} + \chi \e{\theta} + \chi^{2} \e{\varphi} \right ) \nabla \e{\varphi} \right ] \right ) + \mean{\e{\mu}} (\e{\varphi} - \mean{\e{\varphi}}) \\
 &=: \e{z} + \mean{\e{\mu}} (\e{\varphi} - \mean{\e{\varphi}}).
\end{align*}
By the definition of the variable $\e{z}$, for any $\zeta \in H^2_n(\Omega)$, we obtain by integrating by parts and the homogeneous Neumann boundary conditions
\begin{align*}
& \int_{\Omega} \nabla \e{z} \cdot \nabla \zeta \dx = \int_{\Omega} \K \HH \zeta + (\e{\mu} - \mean{\e{\mu}} + \chi \e{\theta} + \chi^{2} \e{\varphi} ) \nabla \e{\varphi} \cdot \nabla \zeta \dx \\
& \quad = \int_{\Omega} \K \HH \zeta - \e{\varphi} \nabla (\e{\mu} - \mean{\e{\mu}} + \chi \e{\theta}) \cdot \nabla \zeta - \left ( \e{\varphi} (\e{\mu} - \mean{\e{\mu}} + \chi \e{\theta}) + \tfrac{1}{2} \chi^{2} \abs{\e{\varphi}}^{2} \right ) \Laplace \zeta \dx.
\end{align*}
Substituting $\zeta = \NN(\e{z})$ so that the left-hand side becomes $\norm{\e{z}}_{L^{2}}^{2}$, and by the elliptic estimate $\norm{\NN(\e{z})}_{H^{2}} \leq C \norm{\e{z}}_{L^{2}}$, this leads to
\begin{align*}
\norm{\e{z}}_{L^{2}}^{2} & \leq \K \norm{\HH}_{L^{2}} \norm{\NN(\e{z})}_{L^{2}} + \norm{\e{\varphi}}_{L^{3}} \norm{\nabla (\e{\mu} - \mean{\e{\mu}} + \chi \e{\theta})}_{L^{2}} \norm{\nabla \NN(\e{z})}_{L^{6}} \\
& \quad + \left ( \norm{\e{\varphi}}_{L^{3}} \norm{\e{\mu} - \mean{\e{\mu}} + \chi \e{\theta}}_{L^{6}} + C\norm{\e{\varphi}}_{L^{4}}^{2} \right ) \norm{\Laplace \NN(\e{z})}_{L^{2}} \\
& \leq C \norm{\e{z}}_{L^{2}} \left ( \K \norm{\HH}_{L^{2}} + \norm{\e{\varphi}}_{L^{3}} \left ( \norm{\nabla \e{\mu}}_{L^{2}} + \chi \norm{\e{\theta}}_{H^{1}} \right ) + \norm{\e{\varphi}}_{L^{4}}^{2} \right ),
\end{align*}
where we have used the Poincar\'{e} inequality to deduce $\norm{\e{\mu} - \mean{\e{\mu}}}_{L^{6}} \leq C \norm{\e{\mu} - \mean{\e{\mu}}}_{H^{1}} \leq C  \norm{\nabla \e{\mu}}_{L^{2}}$.  Then, using $\int_{\Omega} \HH \dx = 0$, we find
\begin{align*}
& \int_{Q_{t}} \HH (\e{p} - \mean{\e{\mu}} \e{\varphi}) \dx \ds = \int_{Q_{t}} \HH (\e{z} + \mean{\e{\mu}} \; \mean{\e{\varphi}}) \dx \ds = \int_{Q_{t}} \HH \e{z} \dx \ds \\
& \quad  \leq C \K \norm{\HH}_{L^{2}(Q)}^{2} + C \int_{0}^{t} \norm{\e{\varphi}}_{L^{3}} \left ( \norm{\nabla \e{\mu}}_{L^{2}} + \norm{\e{\theta}}_{H^{1}} \right ) + \norm{\e{\varphi}}_{L^{4}}^{2} \ds \\
& \quad \leq \frac{1}{8} \norm{m^{\frac{1}{2}}(\e{\varphi}) \nabla \e{\mu}}_{L^{2}(Q_{t})}^{2} + \frac{1}{4} \norm{n^{\frac{1}{2}}(\e{\varphi}) \nabla \e{\theta}}_{L^{2}(Q_{t})}^{2} \\
& \qquad + C \left (1 + \int_{0}^{t} \norm{\e{\theta}}_{L^{2}}^{2} + \norm{\e{\varphi}}_{L^{4}}^{2} \ds \right ).
\end{align*}
Substituting the above estimate into \eqref{RHS:HH} we infer for the source term involving $\HH$ in the energy identity \eqref{CHD:energy:id} the following estimate
\begin{equation}\label{RHS:HH:Est}
\begin{aligned}
& \int_{Q_{t}} \HH \left (\e{p} - \e{\mu} \e{\varphi} - \tfrac{1}{2} \abs{\e{\theta} + \chi \e{\varphi}}^{2} \right ) \dx \ds \\
& \quad \leq C + \frac{1}{4} \norm{m^{\frac{1}{2}}(\e{\varphi}) \nabla \e{\mu}}_{L^{2}(Q_{t})}^{2} + \frac{1}{4} \norm{n^{\frac{1}{2}}(\e{\varphi}) \nabla \e{\theta}}_{L^{2}(Q_{t})}^{2} \\
& \qquad + C \int_{Q_{t}} \Psi(\e{\varphi}) + \abs{\e{\theta}}^{2} \dx \ds,
\end{aligned}
\end{equation}
where we employed \eqref{Psi:lowerbound} and a similar calculation to \eqref{varphiL2:Psi:L1} to deduce
\begin{align*}
\norm{\e{\varphi}}_{L^{4}}^{2} \leq C + \int_{\Omega} \abs{\e{\varphi}}^{q} \dx \leq C \left ( 1 + \int_{\Omega} \Psi(\e{\varphi}) \dx \right ).
\end{align*}
In particular, we note that this motivates the assumption $q \geq 4$ in \eqref{ass:Psi}.  We are now in a position to derive the first uniform estimate.  Let us define the quantity 
\begin{align}\label{defn:GG}
\GG(\e{\varphi}, \e{\theta}) := \frac{1}{2 \eps} \norm{\Psi(\e{\varphi})}_{L^{1}} + \frac{\eps}{2} \norm{\nabla \e{\varphi}}_{L^{2}}^{2} + \norm{\e{\theta}}_{L^{2}}^{2}.
\end{align}
Then, using \eqref{CH:energy:density:LB}, the estimation on the source terms \eqref{RHS:ZYmu:2} and \eqref{RHS:HH:Est}, the boundedness of the initial energy \eqref{ass:initial} and the estimate \eqref{varphiL2:Psi:L1}, we obtain from \eqref{CHD:energy:id} the following integral inequality
\begin{align*}
& \GG(\e{\varphi}(t), \e{\theta}(t)) + \frac{1}{2} \norm{m^{\frac{1}{2}}(\e{\varphi}) \nabla \e{\mu}}_{L^{2}(Q_{t})}^{2} + \frac{1}{2} \norm{n^{\frac{1}{2}}(\e{\varphi}) \nabla \e{\theta}}_{L^{2}(Q_{t})}^{2} + \K \norm{\e{\bm{v}}}_{L^{2}(Q_{t})}^{2} \\
& \quad \leq C + C \int_{0}^{t} \GG(\e{\varphi}(s), \e{\theta}(s)) \ds.
\end{align*}
By virtue of Gronwall's inequality in integral form \cite[Lem. 3.1]{GLNeu} we infer the uniform estimate
\begin{equation}\label{CHD:unif:est:1}
\begin{aligned}
& \frac{1}{\eps} \norm{\Psi(\e{\varphi})}_{L^{\infty}(0,T;L^{1})} + \frac{\eps}{2} \norm{\nabla \e{\varphi}}_{L^{\infty}(0,T;L^{2})}^{2} + \norm{\e{\theta}}_{L^{\infty}(0,T;L^{2})}^{2} \\
& \quad + \norm{m^{\frac{1}{2}}(\e{\varphi}) \nabla \e{\mu}}_{L^{2}(Q)}^{2} + \norm{n^{\frac{1}{2}}(\e{\varphi}) \nabla \e{\theta}}_{L^{2}(Q)}^{2} + \norm{\e{\bm{v}}}_{L^{2}(Q)}^{2} \leq C,
\end{aligned}
\end{equation}
for all $\eps \in (0,\eps_{0}]$, where $\eps_{0}$ is defined in \eqref{defn:eps0}.

\subsection{Second estimates}
Building on \eqref{CHD:unif:est:1} we now derive subsequent uniform estimates.  There exists a positive constant $C_{2}$ such that the Modica--Mortola ansatz function $\e{w} = W(\e{\varphi})$, where $W$ is the bijective function defined in Lemma \ref{lem:ModicaAnsatz:BijW}, satisfies
\begin{align}
\label{CHD:unif:est:2}
\norm{\e{w}}_{L^{\infty}(0,T;W^{1,1})} \leq C_{2}.
\end{align}
Indeed, by \eqref{varphiL2:Psi:L1}, \eqref{CHD:unif:est:1}, and the second inequality of \eqref{Wfunctiondiff}, we infer that
\begin{align*}
\norm{\nabla \e{w}(t)}_{L^{1}} & \leq \int_{\Omega} \frac{1}{\eps} \Psi(\e{\varphi}(t)) + \frac{\eps}{2} \abs{\nabla \e{\varphi}(t)}^{2} \dx \leq C, \\
\norm{\e{w}(t)}_{L^{1}} & = \norm{W(\e{\varphi}(t)) - W(-1)}_{L^{1}} \leq C \int_{\Omega} \abs{\e{\varphi}(t) + 1} (2 + \abs{\e{\varphi}(t)}) \dx \\
& \leq C \left ( 1 + \norm{\e{\varphi}(t)}_{L^{2}}^{2} \right ) \leq C,
\end{align*}
for a.e.~$t \in [0,T]$.  As a further consequence of \eqref{CHD:unif:est:1}, \eqref{Psi:lowerbound} and \eqref{quadraticbound}, we obtain
\begin{align}
\label{CHD:unif:est:3}
\norm{\e{\varphi}(t)}_{L^{q}} \leq C, \quad \norm{\abs{\e{\varphi}(t)}-1}_{L^{2}} \leq C \eps^{\frac{1}{2}} \quad \forall t \in (0,T].
\end{align}

\subsection{H\"{o}lder-in-time uniform estimates}
For any $y := \frac{2q}{q-2} \in (2,4]$, we obtain from \eqref{CHD:varphi} that
\begin{align*}
\norm{\pd_{t} \e{\varphi}}_{L^{2}(0,T;(W^{1,y})')} & \leq \norm{\e{\varphi}}_{L^{\infty}(0,T;L^{q})} \norm{\e{\bm{v}}}_{L^{2}(Q)} + C \left (\norm{\nabla \e{\mu}}_{L^{2}(Q)} + \norm{\e{\varphi}}_{L^{2}(Q)}  \right ).
\end{align*}
Using that $\e{\theta}$ is bounded in $L^{\infty}(0,T;L^{2}(\Omega)) \cap L^{2}(0,T;H^{1}(\Omega)) \subset L^{\frac{10}{3}}(Q)$ and $\e{\varphi}$ is bounded in $L^{\infty}(0,T;L^{q}(\Omega)) \subset L^{\frac{10}{3}}(Q)$ in three spatial dimensions, from \eqref{CHD:theta} we infer that
\begin{align*}
\norm{\pd_{t} \e{\theta}}_{L^{\frac{5}{4}}(0,T;(W^{1,5})')} & \leq \left ( \norm{\e{\theta}}_{L^{\frac{10}{3}}(Q)} + \chi \norm{\e{\varphi}}_{L^{\frac{10}{3}}(Q)} \right )\norm{\e{\bm{v}}}_{L^{2}(Q)}  \\
& \quad + C \left ( \norm{\pd_{t} \e{\varphi}}_{L^{2}(0,T;(W^{1,y})')} + \norm{\nabla \e{\theta}}_{L^{2}(Q)} + \norm{\So}_{L^{2}(Q)} \right ), 
\end{align*}
and altogether this implies for any $ y = \frac{2q}{q-2} \in (2,4]$,
\begin{align}\label{CHD:unif:est:4}
\norm{\pd_{t} \e{\theta}}_{L^{\frac{5}{4}}(0,T;(W^{1,5})')} + \norm{\pd_{t} \e{\varphi}}_{L^{2}(0,T;(W^{1,y})')} \leq C.
\end{align}
We now use the above estimates to obtain H\"{o}lder-in-time bounds for $\e{\varphi}$ and $\e{w}$.  By virtue of the absolute continuity of Banach space-valued $H^{1}(0,T)$ functions \cite[Lem.~2.2.1, p.~44]{Droniou}, we consider \eqref{CHD:varphi} as the following equality in $(W^{1,y})'$: for any $0 \leq \tau < t \leq T$,
\begin{align}\label{Holder:1}
\e{\varphi}(t) - \e{\varphi}(\tau) = \int_{\tau}^{t} \div (m(\e{\varphi}) \nabla \e{\mu}) - \div (\e{\varphi} \e{\bm{v}}) + \U \e{\varphi} \ds.
\end{align}
Let $\phi \in C^{\infty}(\R^{d})$ satisfy $0 \leq \phi \leq 1$ in $\R^{d}$, $\phi = 0$ in $\R^{d} \setminus B_{1}$, where $B_{1}$ is the ball centered at the origin with radius $1$ and $\int_{\R^{d}} \phi \dx = 1$.  Let $\eta_{0}$ denote a small positive number, and for any $\eta \in (0, \eta_{0})$, let $\e{\varphi}_{\eta}$ denote the mollification of $\e{\varphi}$ defined as
\begin{align}\label{mollifier}
\e{\varphi}_{\eta}(x,t) = \int_{B_{1}} \phi(y) \e{\varphi}(x-\eta y, t) \dy = - \eta^{-d} \int_{\R^{d}} \phi \left (\tfrac{(x-z)}{\eta} \right ) \e{\varphi}(z,t) \dz
\end{align}
for $x \in \Omega$, $t \geq 0$ and $\eta \in (0,\eta_{0}]$ and $\e{\varphi}$ has been extended in the exterior neighbourhood $\{x \notin \Omega : \mathrm{dist}(x,\pd \Omega) < \eta \}$ via $\e{\varphi}(x+\eta \bm{\nu}(x),t) = \e{\varphi}(x - \eta \bm{\nu}(x),t)$ for $x \in \pd \Omega$, $\eta \in (0,\eta_{0}]$ and $t \geq 0$.  Then, keeping in mind the following standard properties of mollifiers:
\begin{align}
\label{mollifiers:prop}
\norm{f_{\eta}}_{L^{r}} \leq \norm{f}_{L^{r}} , \quad \norm{\nabla f_{\eta}}_{L^{r}} \leq C \eta^{-1} \norm{f}_{L^{r}} \quad \text{ for any } 1 \leq r < \infty,
\end{align}
we obtain from testing \eqref{Holder:1} with the function $X_{\eta} := \e{\varphi}_{\eta}(t) - \e{\varphi}_{\eta}(\tau)$ 
\begin{align*}
\int_{\Omega} (\e{\varphi}(t) - \e{\varphi}(\tau)) X_{\eta} \dx & = \int_{\Omega} - \nabla X_{\eta} \cdot \left ( \int_{\tau}^{t} -\e{\varphi} \e{\bm{v}} + m(\e{\varphi}) \nabla \e{\mu} \ds \right ) + X_{\eta} \left ( \int_{\tau}^{t} \U \e{\varphi} \ds \right ) \dx \\
& \leq \norm{\e{\varphi}}_{L^{\infty}(0,T;L^{q})} \norm{\e{\bm{v}}}_{L^{2}(Q)} \abs{t - \tau}^{1/2} \norm{\nabla X_{\eta}}_{L^{y}} \\
& \quad + m_{1} \norm{\nabla \e{\mu}}_{L^{2}(Q)} \abs{t-\tau}^{1/2} \norm{\nabla X_{\eta}}_{L^{2}} \\
& \quad + \norm{\U \e{\varphi} }_{L^{2}(Q)} \abs{t - \tau}^{1/2} \norm{X_{\eta}}_{L^{2}}.
\end{align*}
Using \eqref{mollifiers:prop} and the uniform estimate \eqref{CHD:unif:est:3}, we see that
\begin{align*}
\norm{X_{\eta}}_{L^{2}} & \leq \norm{\e{\varphi}(t) - \e{\varphi}(\tau)}_{L^{2}} \leq 2 \norm{\e{\varphi}}_{L^{\infty}(0,T;L^{2})} \leq C, \\
\norm{\nabla X_{\eta}}_{L^{y}} & \leq C \eta^{-1} \norm{\e{\varphi}}_{L^{\infty}(0,T;L^{y})} \leq C \eta^{-1},
\end{align*}
and so
\begin{align*}
\int_{\Omega} (\e{\varphi}(t) - \e{\varphi}(\tau)) X_{\eta} \dx  \leq C \abs{t-\tau}^{1/2} \left ( 1 + \eta^{-1} \right ).
\end{align*}
Together with the following estimate \cite[(3.2)-(3.4)]{Chen96}
\begin{align}\label{mollifier:diff}
\norm{\e{\varphi}_{\eta}(t) - \e{\varphi}(t)}_{L^{2}}^{2} \leq C \norm{\nabla \e{w}(t)}_{L^{1}} \leq C \eta \quad \forall t \in [0,T],
\end{align}
we have
\begin{align*}
\norm{\e{\varphi}(t) - \e{\varphi}(\tau)}_{L^{2}}^{2} \leq C \left ( \eta + \abs{t-\tau}^{1/2}(1+\eta^{-1}) \right ).
\end{align*}
Dividing through by $\abs{t-\tau}^{\beta}$ for some $\beta > 0$ and choosing $\eta = \min(\eta_{0},\abs{t-\tau}^{\gamma})$ for some $\gamma > 0$.  Then, the resulting right-hand side now reads as $C ( \abs{t-\tau}^{\gamma-\beta} + \abs{t-\tau}^{1/2-\beta} + \abs{t-\tau}^{1/2-\gamma-\beta} )$, and for this to be bounded for any $0 \leq \tau < t \leq T$, we require $\frac{1}{2} \geq \beta + \gamma$, $\gamma \geq \beta$.  Choosing $\beta = \gamma = \frac{1}{4}$ and taking the supremum over $0 \leq \tau < t \leq T$ we infer the estimate
\begin{align}\label{CHD:unif:est:5}
\sup_{0 \leq \tau < t \leq T} \frac{\norm{\e{\varphi}(t) - \e{\varphi}(\tau)}_{L^{2}}^{2}}{\abs{t - \tau}^{1/4}} \leq C.
\end{align}
Following the arguments in the proof of \cite[Lem. 3.2]{Chen96} we obtain an analogous estimate for $\e{w}$:
\begin{align}\label{CHD:unif:est:6}
\sup_{0 \leq \tau < t \leq T} \frac{\norm{\e{w}(t) - \e{w}(\tau)}_{L^{1}}}{\abs{t-\tau}^{1/8}} \leq C.
\end{align}

\subsection{Uniform estimate for the spatial mean of $\e{\mu}$}
We now derive a uniform estimate for the mean value $\mean{\e{\mu}}$.  Observe from \eqref{CHD:varphi}, the mean value $\mean{\e{\varphi}}$ satisfies
\begin{align*}
\mean{\e{\varphi}}(t) = u_{0} + \int_{0}^{t} \abs{\Omega}^{-1} \int_{\Omega} \U \e{\varphi} \dx \ds,
\end{align*}
where $u_{0} \in (-1,1)$ is the mean value of the initial condition $\e{\varphi}_{0}$ that is  independent of $\eps$.  Recalling the mollifier $\e{\varphi}_{\eta}$, from the proof of \cite[Lem. 3.4]{Chen96} we infer the following estimate
\begin{align}
\label{mollifier:est}
\norm{\e{\varphi}_{\eta}}_{L^{\infty}} \leq 1 + C \eta^{-\frac{d}{2}} \norm{(\abs{\e{\varphi}}-1)}_{L^{2}} \leq 1 + C \eta^{-\frac{d}{2}} \eps^{\frac{1}{2}},
\end{align}
so that by \eqref{mollifier:diff},
\begin{equation}\label{mean:varphi:est}
\begin{aligned}
\abs{\mean{\e{\varphi}}(t)} & = \abs{u_{0} + \int_{0}^{t} \abs{\Omega}^{-1} \int_{\Omega} \U (\e{\varphi} - \e{\varphi}_{\eta}) + \U \e{\varphi}_{\eta} \dx \ds} \\
 & \leq \abs{u_{0}} + C\norm{\U}_{L^{\infty}} \norm{\e{\varphi} - \e{\varphi}_{\eta}}_{L^{2}(Q_{t})} + t \norm{\U}_{L^{\infty}} \norm{\e{\varphi}_{\eta}}_{L^{\infty}(0,T;L^{\infty}(\Omega))} \\
& \leq \abs{u_{0}} + C \left ( \eta^{\frac{1}{2}} + \eta^{-\frac{d}{2}} \eps^{\frac{1}{2}} \right ) + T \norm{\U}_{L^{\infty}}.
\end{aligned}
\end{equation}
Together with \eqref{cond:T}, this in turn implies that there exists $\eps_{1} > 0$ such that for all $\eps \in (0, \min (\eps_{0},\eps_{1})]$,
\begin{align*}
\int_{\Omega} (\e{\varphi}_{\eta} - \mean{\e{\varphi}_{\eta}}) \e{\varphi} \dx & = \int_{\Omega} (\e{\varphi}_{\eta} - \e{\varphi})\e{\varphi} + (\e{\varphi})^{2} - 1 \dx + \abs{\Omega} (1 - (\mean{\e{\varphi}})^{2}) + \abs{\Omega} \mean{\e{\varphi}} (\mean{\e{\varphi}} - \mean{\e{\varphi}_{\eta}}) \\
& \geq \abs{\Omega} (1 - (\mean{\e{\varphi}})^{2}) - C \norm{\e{\varphi}_{\eta} - \e{\varphi}}_{L^{2}} \norm{\e{\varphi}}_{L^{2}} - \norm{\abs{\e{\varphi}}^{2} - 1}_{L^{1}} \\
& \geq  \abs{\Omega} (1 - (\mean{\e{\varphi}})^{2}) - C(\eta^{\frac{1}{2}} + \eps^{\frac{1}{2}}) \\
& \geq \abs{\Omega} (1 - (\abs{u_{0}} + T \norm{\U}_{L^{\infty}})^{2}) - C \left ( \eta^{\frac{1}{2}} + \eps^{\frac{1}{2}} (1 + \eta^{-\frac{d}{2}}) \right ) > 0,
\end{align*}
if we choose $\eta$ sufficiently small but independent of $\eps$.  In the above we have used \eqref{mollifier:diff} and \eqref{CHD:unif:est:3} to deduce that
\begin{align*}
\abs{\mean{\e{\varphi}} (\mean{\e{\varphi}} - \mean{\e{\varphi}_{\eta}})} \leq C \norm{\e{\varphi}}_{L^{2}} \norm{\e{\varphi} - \e{\varphi}_{\eta}}_{L^{2}} \leq C \eta^{\frac{1}{2}}, \\
\norm{\abs{\e{\varphi}}^{2} - 1}_{L^{1}} \leq \norm{\abs{\e{\varphi}+1} \abs{ \abs{\e{\varphi}}-1} }_{L^{1}} \leq C \norm{\abs{\abs{\e{\varphi}}-1}}_{L^{2}} \leq C \eps^{\frac{1}{2}}.
\end{align*}
Then, returning to the equality \eqref{RHS:mean:mu} and this time choose $\psi$ as the solution to
\begin{align*}
- \Laplace \psi = -(\e{\varphi}_{\eta} - \mean{\e{\varphi}_{\eta}}) \text{ in } \Omega, \quad \pdnu \psi = 0 \text{ on } \pd \Omega \quad \text{ with } \mean{\psi} = 0,
\end{align*}
as in the proof of \cite[Lem. 3.4]{Chen96}, with
\begin{align*}
\norm{\psi}_{C^{2}(\Omega)} \leq C \norm{\e{\varphi}_{\eta}}_{C^{1}(\Omega)} \leq C \eta^{-1} \left ( 1 + \eps^{\frac{1}{2}} \eta^{-\frac{d}{2}} \right ),
\end{align*}
we find that
\begin{align*}
\abs{\mean{\e{\mu}}(t)} \leq C \norm{\psi}_{C^{2}(\Omega)} \left ( 1 + \norm{\nabla \e{\mu}(t)}_{L^{2}} + \norm{\nabla \e{\theta}(t)}_{L^{2}} \right ),
\end{align*}
i.e., $\mean{\e{\mu}}$ is bounded in $L^{2}(0,T)$ and Poincar\'{e}'s inequality yields 
\begin{align}\label{CHD:unif:est:7}
\norm{\e{\mu}}_{L^{2}(Q)} \leq C.
\end{align}

\subsection{Uniform estimates for the pressure}

It remains to derive uniform estimates for the pressure.  From \eqref{Pressure:poisson}, the pressure $\e{p}$ satisfies
\begin{align*}
\int_{\Omega} \nabla \e{p} \cdot \nabla \zeta \dx = \int_{\Omega} \K \HH \zeta - (\e{\mu} + \chi \e{\theta} + \chi^{2} \e{\varphi}) \nabla \e{\varphi} \cdot \nabla \zeta \dx
\end{align*}
for all $\zeta \in H^{1}(\Omega)$.  Choose $\zeta = \NN(\e{p}) \in H^2_n \cap L^{2}_{0}$, integrating by parts and using the elliptic estimate $\norm{\NN(\e{p})}_{H^{2}} \leq C \norm{\e{p}}_{L^{2}}$ yields
\begin{equation}\label{pressure:est}
\begin{aligned}
\norm{\e{p}}_{L^{2}}^{2} & = \int_{\Omega} \K \HH \NN(\e{p}) + \e{\varphi} \nabla (\e{\mu} + \chi \e{\theta}) \cdot \nabla \NN(\e{p}) \dx \\
& \quad + \int_{\Omega} \e{\varphi} (\e{\mu} + \chi \e{\theta}) \Laplace \NN(\e{p}) + \tfrac{1}{2} \chi^{2} \abs{\e{\varphi}}^{2} \Laplace \NN(\e{p}) \dx \\
& \leq C \norm{\e{p}}_{L^{2}} \left ( 1 + \norm{\e{\varphi}}_{L^{3}} \left ( \norm{\e{\mu}}_{H^{1}} + \norm{\e{\theta}}_{H^{1}} \right ) + \norm{\e{\varphi}}_{L^{4}}^{2} \right ).
\end{aligned}
\end{equation}
Then, using \eqref{CHD:unif:est:1}, \eqref{CHD:unif:est:2} and \eqref{CHD:unif:est:6} we have
\begin{align}\label{CHD:unif:est:8}
\norm{\e{p}}_{L^{2}(Q)} \leq C.
\end{align}

\section{Compactness}\label{sec:Comp}
For each $\eps \in (0,\eps_{2}]$, where $\eps_{2} = \min(\eps_{0},\eps_{1})$, $\eps_{0}$ is defined in \eqref{defn:eps0} and $\eps_{1}$ is the constant from the derivation of \eqref{CHD:unif:est:7}, let $(\e{\varphi}, \e{\mu}, \e{\theta}, \e{\bm{v}}, \e{p})$ denote a solution to \eqref{CHD}.  Then, by the uniform estimates \eqref{CHD:unif:est:1}-\eqref{CHD:unif:est:4}, \eqref{CHD:unif:est:5}, \eqref{CHD:unif:est:6}, \eqref{CHD:unif:est:7}, \eqref{CHD:unif:est:8}, we immediately deduce the existence of functions $\theta, \mu \in L^{2}(0,T;H^{1})$, $\bm{v} \in L^{2}(Q)^{d}$ and $p \in L^{2}(0,T;L^{2}_{0})$ such that the convergence statements \eqref{mu:theta:L2H1}, \eqref{theta:LinftyL2}, \eqref{v:L2L2} and \eqref{p:L2L2} hold along a subsequence $\{\eps_{j}\}_{j \in \N}$ converging to zero.  We claim additionally that
\begin{align*}
\theta^{\eps_{j}} \longrightarrow \theta \text{ strongly in } C^{0}_{\mathrm{w}}([0,T];L^{2}(\Omega)) \cap L^{2}(Q),
\end{align*}
which will allow us to attain the initial condition for $\theta$.  The strong convergence in $L^{2}(Q)$ follows from the application of \cite[\S~8, Cor.~4]{Simon86} and the uniform boundedness in $L^{2}(0,T;H^{1}(\Omega)) \cap W^{1,\frac{5}{4}}(0,T;(W^{1,5}(\Omega))')$.  For the strong convergence in $C^{0}_{\mathrm{w}}([0,T];L^{2}(\Omega))$, we argue as in \cite[(3.119)]{Feireisl}.  Let $\phi \in C^{\infty}_{c}(\Omega)$ and define $f_{j}(t) : [0,T] \to \R$ by $f_{j}(t) := \int_{\Omega} \theta^{\eps_{j}}(t) \phi \dx$.  Boundedness of $\theta^{\eps_{j}}$ in $L^{\infty}(0,T;L^{2}(\Omega))$ implies that $f_{j}$ is bounded uniformly for every $t \in [0,T]$, and by the boundedness of $\pd_{t} \theta^{\eps_{j}}$ in $L^{\frac{5}{4}}(0,T;(W^{1,5}(\Omega))')$ one observes that
\begin{align*}
\abs{f_{j}(t) - f_{j}(\tau)} & \leq \abs{\biginner{\int_{\tau}^{t} \pd_{t} \theta^{\eps_{j}}(s) \ds}{\phi}_{W^{1,5}}} \leq \norm{\phi}_{W^{1,5}} \norm{\pd_{t} \theta^{\eps_{j}}}_{L^{\frac{5}{4}}(0,T;(W^{1,5})')} \abs{t - \tau}^{\frac{1}{5}} \\
& \leq C \abs{t - \tau}^{\frac{1}{5}},
\end{align*}
with a constant $C$ not depending on $j$.  Hence $\{f_{j}\}_{j}$ is also equicontinuous and by the Arzel\`{a}--Ascoli theorem
\begin{align*}
\int_{\Omega} \theta^{\eps_{j}}(t) \phi \dx \to \int_{\Omega} \theta(t) \phi \dx \text{ strongly in } C^{0}([0,T]) \quad \forall \phi \in C^{\infty}_{c}(\Omega).
\end{align*}
Using a density argument for $\phi \in L^{2}(\Omega)$ yields the desired assertion.  We now infer some compactness for $\{\varphi^{\eps_{j}}\}_{j \in \N}$ and $\{w^{\eps_{j}}\}_{j \in \N}$.

\begin{lem}\label{lem:CH:weakcompactness}
Let $0 < \alpha < \frac{1}{8}$.  There exists a measurable set $\Omega_{+} \subset Q$ of finite perimeter, i.e., $\chara_{\Omega_{+}(t)} \in L^{\infty}(0,T;BV(\Omega, \{0,1\}))$ such that
\begin{align*}
w^{\eps_{j}} \longrightarrow w := \chara_{\Omega_{+}} & \text{ weakly* in } L^{\infty}(0,T;W^{1,1}(\Omega)), \\
w^{\eps_{j}} \longrightarrow \chara_{\Omega_{+}} & \text{ strongly in } C^{0,\alpha}(0,T;L^{1}(\Omega)) \cap L^{1}(Q) \text{ and a.e. in } Q, \\
\varphi^{\eps_{j}} \longrightarrow \varphi := -1 + 2 \chara_{\Omega_{+}} & \text{ strongly in } C^{0,\alpha}([0,T];L^{2}(\Omega)) \cap L^{2}(Q) \text{ and a.e. in } Q.
\end{align*}
Furthermore, there exists a positive constant $C$ such that for all $0 \leq \tau < t \leq T$,
\begin{align*}
\int_{\Omega} \abs{\chara_{\Omega_{+}(t)} - \chara_{\Omega_{+}(\tau)}} \dx \leq C \abs{t - \tau}^{\frac{1}{4}}, \quad \norm{\nabla \chara_{\Omega_{+}(t)}}_{\M^{d}} = \abs{\nabla \chara_{\Omega_{+}(t)}}(\Omega) \leq C,
\end{align*}
that is, $\chara_{\Omega_{+}} \in L^{\infty}(0,T;BV(\Omega, \{0,1\})) \cap C^{0,\frac{1}{4}}([0,T];L^{1}(\Omega))$.
\end{lem}

\begin{proof}
The estimate \eqref{CHD:unif:est:6} shows that the time translation of $\e{w}$ satisfies
\begin{align*}
\norm{\tau_{h} \e{w} - \e{w}}_{L^{1}(0,T-h;L^{1})} \leq C \abs{h}^{\frac{1}{8}} (T-h) \to 0 \text{ as } h \to 0,
\end{align*}
uniformly in $\eps$.  By the compact embedding $W^{1,1}(\Omega) \subset \subset L^{1}(\Omega)$, for every $0 \leq t_{1} < t_{2} \leq T$, the set $\{\int_{t_{1}}^{t_{2}} \e{w}(s) \ds : \eps \in (0,\eps_{2}]\}$ is relatively compact in $L^{1}(\Omega)$.  Then, together with \eqref{CHD:unif:est:2}, Lemma \ref{AubinLionHolder} and \cite[\S~6, Thm.~3]{Simon86} imply that there exists a subsequence $\{\eps_{j}\}_{j \in \N}$ and a function $w$ such that for any $\alpha \in (0,\frac{1}{8})$,
\begin{align*}
w^{\eps_{j}} \rightarrow w \text{ in } C^{0,\alpha}([0,T];L^{1}(\Omega)) \cap L^{1}(Q) \text{ and a.e. in } Q.
\end{align*}
Let $\varphi$ be the function defined by the relation $w(x,t) = W(\varphi(x,t))$.  This is well-defined as $W(\cdot)$ is bijective.  Then, by the first inequality of \eqref{Wfunctiondiff}, we obtain 
\begin{align*}
C_{1} \abs{\varphi^{\eps_{j}}(x,t) - \varphi(x,t)}^{2} \leq \abs{w^{\eps_{j}}(x,t) - w(x,t)} \text{ for a.e.~} (x,t) \in Q.
\end{align*}
Using the convergence results for $\{ w^{\eps_{j}}\}_{j \in \N}$ this then yields that $\varphi^{\eps_{j}} \to \varphi$ a.e.~in $Q$ and strongly in $L^{2}(Q)$.  To show that $\varphi^{\eps_{j}} \to \varphi$ in $C^{0,\alpha}([0,T];L^{2}(\Omega))$ with Lemma \ref{AubinLionHolder}, it suffices to show that, for any $0 \leq t_{1} < t_{2} \leq T$, the set $F := \{ \int_{t_{1}}^{t_{2}} \varphi^{\eps_{j}}(t) \dt : j \in \N \}$ is relatively compact in $L^{2}(\Omega)$.  The estimate \eqref{CHD:unif:est:3} implies that $\{\e{\varphi}\}_{\eps \in (0,\eps_{2}]}$ is bounded in $L^{\infty}(0,T;L^{q}(\Omega))$ for $q \geq 4$, then by H\"{o}lder's inequality,
\begin{equation}
\label{relativecompactnessL2}
\begin{aligned}
\bignorm{ \int_{t_{1}}^{t_{2}} \varphi^{\eps_{j}}(\cdot,t) \dt}_{L^{q}(\Omega)}^{q} & \leq \abs{t_{2} - t_{1}}^{q-1} \norm{\varphi^{\eps_{j}}}_{L^{q}(\Omega \times (t_{1}, t_{2}))}^{q}\\
&  \leq T^{q} \norm{\e{\varphi}}_{L^{\infty}(0,T;L^{q})}^{q} \leq c,
\end{aligned}
\end{equation}
for some positive constant $c$.  So, $F$ is bounded in $L^{q}(\Omega) \subset L^{2}(\Omega)$.  The claim then follows from the application of the Kolmogorov--Riesz compactness theorem \cite[Thm.~2.32]{Adams} once we show that, for all $\delta > 0$, there exist $\eta > 0$ and a subset $V$ such that for all $f \in F$, $a \in \R^{d}$ with $\abs{a} < \eta$, 
\begin{align*}
\int_{\Omega} \abs{f(x+a) - f(x)}^{2} \dx < \delta \quad \text{ and } \quad \int_{\Omega \setminus V} \abs{f(x)}^{2} \dx < \delta.
\end{align*}
Here, we have extended all $f \in F$ by zero outside $\Omega$ and retained the same notation.  For the former condition, we choose 
\begin{align*}
\eta < \frac{C_{1} \delta}{T^{2} C_{2}},
\end{align*}
where $C_{1}$ is the constant in \eqref{Wfunctiondiff} and $C_{2}$ is the constant in \eqref{CHD:unif:est:2}.  Then, by H\"{o}lder's inequality, the first inequality of \eqref{Wfunctiondiff}, Fubini's theorem, absolute continuity on lines for $W^{1,1}$ functions, we obtain (suppressing the dependence on $t$),
\begin{align*}
&  \int_{\Omega} \abs{\int_{t_{1}}^{t_{2}} \varphi^{\eps_{j}}(x+a) - \varphi^{\eps_{j}}(x) \dt}^{2} \dx  \leq \abs{t_{2} - t_{1}} \int_{t_{1}}^{t_{2}} \int_{\Omega} \abs{\varphi^{\eps_{j}}(x+a) - \varphi^{\eps_{j}}(x)}^{2} \dx \dt \\
& \leq \frac{T}{C_{1}} \int_{t_{1}}^{t_{2}} \int_{\Omega} \abs{w^{\eps_{j}}(x+a) -w^{\eps_{j}}(x)} \dx \dt \leq \frac{T}{C_{1}} \int_{0}^{1} \int_{t_{1}}^{t_{2}} \int_{\Omega} \abs{\nabla w^{\eps_{j}}(x + a \xi)} \abs{a} \dx \dt \dd \xi \\
& \leq \eta \frac{T^{2}}{C_{1}} \norm{\nabla w^{\eps_{j}}}_{L^{\infty}(0,T;L^{1})} < \delta.
\end{align*}
For the latter condition, we choose $V$ to be any compact subset of $\Omega$ such that its measure satisfies
\begin{align*}
\abs{\Omega \setminus V} < \delta^{\frac{q}{q-2}} c^{-\frac{2}{q-2}},
\end{align*}
where $c$ is the constant in \eqref{relativecompactnessL2}.  Then, by H\"{o}lder's inequality,
\begin{align*}
\int_{\Omega \setminus V} \abs{\int_{t_{1}}^{t_{2}} \varphi^{\eps_{j}}(x,t) \dt}^{2} \dx \leq \bignorm{\int_{t_{1}}^{t_{2}} \varphi^{\eps_{j}}(t) \dt}_{L^{q}(\Omega)}^{2} \abs{\Omega \setminus V}^{\frac{q-2}{q}} < \delta .
\end{align*}
By the embedding $W^{1,1}(\Omega) \subset BV(\Omega)$, the limit function $w$ belongs to $L^{\infty}(0,T;BV(\Omega))$.  Then, applying Fatou's lemma to the second estimate of \eqref{CHD:unif:est:3} for $\e{\varphi}$ shows that the limit function $\varphi$ satisfies $\abs{\varphi(x,t)} = 1$.  Using the definition of the function $W$ in Lemma \ref{lem:ModicaAnsatz:BijW} and the first identity of \eqref{Psi:basic}, it holds that $w(t) \in BV(\Omega, \{0,1\})$ for all $t \in [0,T]$.  Defining the following measurable sets
\begin{align*}
\Omega_{+}(t) := \left \{ x \in \Omega : \lim_{\delta \to 0} \frac{1}{\abs{B_{\delta}(x)}} \int_{B_{\delta}(x)} w(y,t) \dy = 1 \right \}, \quad \Omega_{-}(t) := \Omega \setminus \Omega_{+}(t)
\end{align*}
for any $t \in [0,T]$ allows us to consider $w(x,t)$ as the characteristic function $\chara_{\Omega_{+}(t)}(x)$ of $\Omega_{+}$.  Furthermore, the set $\Omega_{+} \subset Q$  has finite perimeter, as its characteristic function belongs to the space $BV(\Omega, \{0,1\})$.  Then, by the relation $\varphi = W^{-1}(w)$ we obtain the assertion that $\varphi = -1 + 2 \chara_{\Omega_{+}(t)}$.  

The assertion regarding the total variation $\abs{\nabla \chara_{\Omega_{+}(t)}}(\Omega)$  follows from applying the weak lower semicontinuity of the BV-norm to the estimate for $\nabla w^{\eps_{j}}$ in \eqref{CHD:unif:est:2}.  Meanwhile, arguing as in \cite{Chen96}, by \eqref{CHD:unif:est:5} and the strong convergence of $\varphi^{\eps_{j}}$ to $\varphi$ in $C^{0}([0,T];L^{2}(\Omega))$, it holds that
\begin{align*}
\norm{\chara_{\Omega_{+}(t)} - \chara_{\Omega_{+}(\tau)}}_{L^{1}} & = \norm{\chara_{\Omega_{+}(t)} - \chara_{\Omega_{+}(\tau)}}_{L^{2}}^{2} = \frac{1}{4} \norm{\varphi(t) - \varphi(\tau)}_{L^{2}}^{2} \\
& = \lim_{\eps_{j} \to 0} \frac{1}{4} \norm{\varphi^{\eps_{j}}(t) - \varphi^{\eps_{j}}(\tau)}_{L^{2}}^{2} \leq C \abs{t-\tau}^{\frac{1}{4}}.
\end{align*}
The proof is complete.
\end{proof}

We define $\e{e}(\e{\varphi})$ as the Ginzburg--Landau density and $\e{\xi}(\e{\varphi})$ as the discrepancy density by
\begin{align*}
\e{e}(\e{\varphi}) := \frac{1}{\eps} \Psi(\e{\varphi}) + \frac{\eps}{2} \abs{\nabla \e{\varphi}}^{2}, \quad \e{\xi}(\e{\varphi}) = \frac{\eps}{2} \abs{\nabla \e{\varphi}}^{2} - \frac{1}{\eps} \Psi(\e{\varphi}).
\end{align*}
Then, the uniform estimate \eqref{CHD:unif:est:1} implies
\begin{equation}\label{BddRadon}
\begin{aligned}
\norm{\e{e}(\e{\varphi})}_{L^{\infty}(0,T;L^{1})} +
\norm{\eps \nabla \e{\varphi} \otimes \nabla \e{\varphi}}_{L^{\infty}(0,T;L^{1})} \leq C.
\end{aligned}
\end{equation}
For $1 \leq i, k \leq d$, we introduce the linear functionals on $C^{0}_{c}(Q)$:
\begin{align}\label{defn:lambda:meas}
\inner{\e{\lambda}}{Y} := \int_{Q} \e{e}(\e{\varphi}) \, Y \dx \dt, \quad \inner{\e{\lambda}_{ik}}{Y} := \int_{Q} \eps \pd_{x_{i}} \e{\varphi} \pd_{x_{k}} \e{\varphi} \, Y \dx \dt,
\end{align}
for $Y \in C^{0}_{c}(Q)$.  The integrals are well-defined as $\e{e}(\e{\varphi})$ and $\eps \pd_{x_{i}} \e{\varphi} \pd_{x_{k}} \e{\varphi}$ are $L^{1}(Q)$-functions.  Hence we can interpret $\e{\lambda}$ and $\e{\lambda}_{ik}$, for $1 \leq i,k \leq d$, as Radon measures. 

\begin{lem}\label{lem:meas+energyIneq}
There exist a subsequence $\{\eps_{j}\}_{j \in \N}$ converging to zero, Radon measures $\lambda(x,t)$ and $\{\lambda_{ik}(x,t)\}_{1 \leq i,k \leq d}$ on $Q$, Radon measures $\lambda^{t}(x)$ and $\{\lambda_{ik}^{t}(x)\}_{1 \leq i,k \leq d}$ on $\Omega$ for a.e.~$t \in (0,T)$ such that
\begin{equation}\label{compactness:meas}
\begin{aligned}
\lambda^{\eps_{j}} \to \lambda(x,t) \text{ weakly* in } \mathcal{M}(Q), \quad & \dd \lambda(x,t) = \dd \lambda^{t}(x) \dt, \\
\lambda_{ik}^{\eps_{j}} \to \lambda_{ik}(x,t) \text{ weakly* in } \mathcal{M}(Q), \quad & \dd \lambda_{ik}(x,t) = \dd \lambda_{ik}^{t}(x) \dt.
\end{aligned}
\end{equation}
Furthermore, 
\begin{align}\label{lscBV}
\abs{ \nabla \chara_{\Omega_{+}(t)}}(\Omega)  \leq \lambda^{t}(\Omega),
\end{align} 
and there exists a function $E(t)$ such that the energy \eqref{CHD:energy} satisfies
\begin{align}\label{CH:Limit:Et}
\mathcal{E}(\varphi^{\eps_{j}}(t), \theta^{\eps_{j}}(t)) \to E(t) := \lambda^{t}(\Omega) + \int_{\Omega} \frac{1}{2} \abs{\theta(t)}^{2} \dx - \frac{\chi^{2}}{2} \abs{\Omega}
\end{align}
and for a.e.~$0 \leq \tau \leq t \leq T$ with the notation $\varphi(x,s) := -1 + 2 \chara_{\Omega_{+}(s)}(x)$, we have
\begin{equation}\label{CH:Limit:energy:ineq}
\begin{aligned}
& E(t) + \int_{\tau}^{t} \int_{\Omega} m \left (\varphi \right ) \abs{\nabla \mu}^{2} + n \left (\varphi \right ) \abs{\nabla \theta}^{2} + \K \abs{\bm{v}}^{2} \dx \ds \\
& \quad \leq E(\tau) + \int_{\tau}^{t} \int_{\Omega} \U \varphi \mu + \So \theta + \HH \left ( p - \varphi \mu - \tfrac{1}{2} \abs{\theta + \chi \varphi}^{2} \right ) \dx \ds.
\end{aligned}
\end{equation}

\end{lem}

\begin{proof}
The estimate \eqref{BddRadon} and the compactness of Radon measures yield the existence of $\lambda$ and $\lambda_{ij}$.  The decomposition of $\lambda$ into a spatial component $\lambda^{t} \in \mathcal{M}(\Omega)$ and a time component follows from the application of the disintegration theorem, see for example \cite[Proof of Prop.~3.15]{Kwak}.  In particular, the $L^{\infty}$-boundedness in time from \eqref{BddRadon} ensures the limit measure $\lambda$ is absolutely continuous in time, i.e.,
\begin{align*}
\lambda(A \times I) \to 0 \text{ whenever } \abs{I} \to 0 \text{ with measurable } I \subset [0,T]
\end{align*}
for all measurable $A \subset \Omega$, cf. \cite[p.~407]{AbelsLengeler}.  The same assertion also applies to the signed measures $\{\lambda_{ik}\}_{1 \leq i,k \leq d}$ by repeating the procedure for the positive and negative parts respectively.  

Recalling the definition of the Modica--Mortola ansatz $w^{\eps_{j}} = W(\varphi^{\eps_{j}})$, by Young's inequality, one obtains for a.e.~$t \in (0,T)$,
\begin{align*}
\abs{\nabla w^{\eps_{j}}(t)}(\Omega) = \norm{\nabla w^{\eps_{j}}(t)}_{L^{1}(\Omega)} \leq \norm{e^{\eps_{j}}(\varphi^{\eps_{j}}(t))}_{L^{1}(\Omega)}.
\end{align*}
Passing to the limit $j \to \infty$ and using the lower semicontinuity of the BV-norm yields \eqref{lscBV}.  By \eqref{varphiL2:Psi:L1} and \eqref{CHD:unif:est:1}, the function
\begin{align}\label{CH:EE:form}
\EE(\e{\varphi}(t), \e{\theta}(t)) = \int_{\Omega} \e{e}(\e{\varphi}(t)) + \frac{1}{2} \abs{\e{\theta}(t)}^{2} - \frac{\chi^{2}}{2} \abs{\e{\varphi}(t)}^{2} \dx
\end{align}
is bounded uniformly in $\eps$ for a.e.~$t \in [0,T]$, and thus we can define a pointwise limit (along subsequences)
\begin{align*}
E(t) := \lim_{j \to \infty} \EE(\varphi^{\eps_{j}}(t), \theta^{\eps_{j}}(t))\text{ for a.e. } t \in [0,T].
\end{align*}
Using a similar derivation to the energy identity \eqref{CHD:energy:id}, for any $0 \leq \tau < t \leq T$, it holds that
\begin{equation}\label{CH:epsj:energy:ineq}
\begin{aligned}
& \int_{\tau}^{t} \int_{\Omega} m(\varphi^{\eps_{j}} ) \abs{\nabla\mu^{\eps_{j}} }^{2} + n(\varphi^{\eps_{j}} ) \abs{\nabla \theta^{\eps_{j}} }^{2} + \K \abs{\bm{v}^{\eps_{j}}}^{2} \dx \ds \\
& \quad  = \int_{\tau}^{t} \int_{\Omega} \U \varphi^{\eps_{j}} \mu^{\eps_{j}} + \So \theta^{\eps_{j}} + \HH \left ( p^{\eps_{j}} - \varphi^{\eps_{j}} \mu^{\eps_{j}} - \tfrac{1}{2} \abs{\theta^{\eps_{j}} + \chi \varphi^{\eps_{j}}}^{2} \right ) \dx \ds \\
& \qquad + \EE(\varphi^{\eps_{j}}(\tau), \theta^{\eps_{j}}(\tau)) - \EE(\varphi^{\eps_{j}}(t), \theta^{\eps_{j}}(t)).
\end{aligned}
\end{equation}
For fixed $0 \leq \tau < t \leq T$ such that $E(t)$ and $E(\tau)$ are defined, passing to the limit $j \to \infty$ in the above equality and employ the weak/strong convergences leads to
\begin{align*}
\int_{\tau}^{t} \int_{\Omega} \U \varphi \mu + \So \theta + \HH \left ( p - \varphi \mu - \tfrac{1}{2} \abs{\theta + \chi \varphi}^{2} \right ) \dx \ds + E(\tau) - E(t)
\end{align*}
for the right-hand side.  Meanwhile, the boundedness and continuity of the mobility $m$, a.e.~convergence of $\varphi^{\eps_{j}}$ to $\varphi$ in $Q$ and Lebesgue's dominating convergence theorem yields that $m^{\frac{1}{2}}(\varphi^{\eps_{j}}) \bm{\zeta} \to m^{\frac{1}{2}}(\varphi) \bm{\zeta}$ strongly in $L^{2}(Q)^{d}$ for any $\bm{\zeta} \in L^{2}(Q)^{d}$.  Hence, together with the weak convergence of $\nabla \mu^{\eps_{j}}$ we obtain that $m^{\frac{1}{2}}(\varphi^{\eps_{j}}) \nabla \mu^{\eps_{j}} \to m^{\frac{1}{2}}(\varphi) \nabla \mu$ weakly in $L^{2}(Q)^{d}$.  A similar argument shows $n^{\frac{1}{2}}(\varphi^{\eps_{j}}) \nabla \theta^{\eps_{j}} \to n^{\frac{1}{2}}(\varphi) \nabla \theta$ weakly in $L^{2}(Q)^{d}$.  The inequality \eqref{CH:Limit:energy:ineq} follows from applying the weak lower semicontinuity of the $L^{2}(Q)$-norm to the left-hand side of \eqref{CH:epsj:energy:ineq}.

It remains to show the explicit form for the limit $E(t)$ in \eqref{CH:Limit:Et}.  For this purpose, consider testing \eqref{CH:EE:form} with an arbitrary test function $\delta \in C^{\infty}_{c}(0,T)$ and then passing to the limit $j \to \infty$, yielding 
\begin{equation}\label{CH:EE:identify}
\begin{aligned}
\int_{0}^{T} \delta(t) \lambda^{t}(\Omega) \dt &  = 
\lim_{j \to \infty} \int_{Q} \delta(t) e^{\eps_{j}}(\varphi^{\eps_{j}}) \dx \dt \\
& = \int_{0}^{T} \delta(t) \left ( E(t) - \int_{\Omega} \frac{1}{2} \abs{\theta(t)}^{2} - \frac{\chi^{2}}{2} \abs{\varphi(t)}^{2} \dx \right ) \dt.
\end{aligned}
\end{equation}
In the above we have used the strong convergence of $\theta^{\eps_{j}}$ in $L^{2}(Q)$ to show that
\begin{equation}\label{ptl:squareterms}
\begin{aligned}
\abs{\int_{Q} \delta \left ( \abs{\theta^{\eps_{j}}}^{2} - \abs{\theta}^{2} \right ) \dx \dt } & \leq \norm{\theta^{\eps_{j}} + \theta}_{L^{2}(Q)}  \norm{\theta^{\eps_{j}} - \theta}_{L^{2}(Q)} \norm{\delta}_{L^{\infty}(0,T)} \\
& \leq C \norm{\theta^{\eps_{j}} - \theta}_{L^{2}(Q)} \to 0,
\end{aligned}
\end{equation}
and a similar argument can be applied to pass to the limit for the term $\delta \abs{\varphi^{\eps_{j}}}^{2}$.  Note that $\abs{\varphi(x,t)}^{2} = \abs{-1 + 2 \chara_{\Omega_{+}(t)}(x)}^{2} = 1$ for a.e.~$(x,t) \in Q$, and so this yields the expression \eqref{CH:Limit:Et}.
\end{proof}

\begin{remark}
Due to the presence of the source terms in the energy identity \eqref{CH:Limit:energy:ineq}, the ``sharp interface'' energy $E(t)$ need not be monotone, compared to previous studies in the literature.
\end{remark}

\begin{lem}\label{lem:meas:varifold}
There exist $\lambda$-measurable functions $\{\pi_{l}(x,t)\}_{1 \leq l \leq d}$, and $\lambda$-measurable unit vectors $\{\bm{\nu}_{l}\}_{1 \leq l \leq d}$ such that $0 \leq \pi_{l} \leq 1$, $\sum_{l=1}^{d} \pi_{l} \leq 1$, $\sum_{k=1}^{d} \bm{\nu}_{k} \otimes \bm{\nu}_{k} = \id$ $\lambda$-a.e.~in $Q$, and the symmetric, positive semi-definite matrix $\bm{\Upsilon} = (\Upsilon_{ik})_{1 \leq i,k \leq d}$ defined as 
\begin{align}\label{Upsilon:form}
\bm{\Upsilon} := \pi_{1} \bm{\nu}_{1} \otimes \bm{\nu}_{1} + \dots + \pi_{d} \bm{\nu}_{d} \otimes \bm{\nu}_{d} \quad \lambda-\text{a.e. in } Q
\end{align}
satisfies
\begin{align}\label{lambda:relation}
\dd \lambda_{ik}(x,t) = \Upsilon_{ik}(x,t) \dd \lambda(x,t) \quad \lambda-\text{a.e. in } Q.
\end{align}
\end{lem}

\begin{proof}
The proof can be found in \cite[\S~3.5]{Chen96} and \cite[\S~3.2.7]{Kwak}, and so we briefly sketch the details.  The main point is that the crucial result \cite[Thm.~3.6]{Chen96} on the non-negativity of the discrepancy measure $\xi^{\eps_{j}}(\varphi^{\eps_{j}})$ in the limit $j \to \infty$ depends only on the form of equation \eqref{CHD:mu} and can be applied in our present setting, as $\mu^{\eps_{j}} + \chi \theta^{\eps_{j}} + \chi^{2} \varphi^{\eps_{j}} \in L^{2}(\Omega)$ for a.e.~$t \in (0,T)$.  This shows that 
\begin{align}\label{abscts:lambda}
\int_{Q} \bm{Y} \otimes \bm{Z} : (\dd \lambda_{ik})_{1 \leq i,k \leq d} \leq \int_{Q} \abs{ \bm{Y}} \abs{\bm{Z}} \dd \lambda \quad \forall \bm{Y}, \bm{Z} \in C^{0}_{c}(Q; \R^{d}),
\end{align}
which implies the measures $\{\lambda_{ik}\}_{ 1 \leq i,k \leq d}$ are absolutely continuous with respect to $\lambda$, and leads to the existence of $\lambda$-measurable functions $\Upsilon_{ik}$ such that \eqref{lambda:relation} holds.  The symmetry and positive semi-definiteness of $\bm{\Upsilon}$ are inherited from the matrix $\eps_{j} \nabla \varphi^{\eps_{j}} \otimes \nabla \varphi^{\eps_{j}}$.  Then, there exists an orthonormal basis $\{\bm{\nu}_{l}\}_{1 \leq l \leq d}$ composed of eigenvectors of $\bm{\Upsilon}$ with corresponding eigenvalues $\{\pi_{l}\}_{1 \leq l \leq d}$ such that $\bm{\Upsilon}$ can be expressed as in \eqref{Upsilon:form}.  Due to \eqref{abscts:lambda}, it holds that $0 \leq (\bm{\Upsilon} \bm{\zeta}) \cdot \bm{\zeta} \leq \abs{\bm{\zeta}}^{2}$ and so no eigenvalues can be greater than $1$, while the assertion $\sum_{l=1}^{d} \pi_{l} \leq 1$ can be obtained from passing to the limit $j \to \infty$ in the following inequality:
\begin{align*}
\int_{Q} Y \, \tr{ \eps_{j} \nabla \varphi^{\eps_{j}} \otimes \nabla \varphi^{\eps_{j}} } \dx \dt \leq \int_{Q} \abs{Y} \, \left ( e^{\eps_{j}}(\varphi^{\eps_{j}}) + \xi^{\eps_{j}}(\varphi^{\eps_{j}}) \right ) \dx \dt,
\end{align*}
for $Y \in C^{0}(Q)$, leading to (using that $\{ \bm{\nu}_{l} \}_{1 \leq l \leq d}$ are orthonormal and so $\tr { \bm{\nu}_{k} \otimes \bm{\nu}_{k} } = \abs{\bm{\nu}_{k}}^{2} = 1$)
\begin{align*}
\int_{Q} Y \sum_{l=1}^{d} \Upsilon_{ll} \dd \lambda(x,t) = \int_{Q} Y \sum_{l=1}^{d} \pi_{l} \dd \lambda(x,t) \leq \int_{Q} \abs{Y} \dd \lambda(x,t),
\end{align*} 
i.e., $\sum_{l=1}^{d} \pi_{l} \dd \lambda \leq \dd \lambda$ and so $\sum_{l=1}^{d} \pi_{l} \leq 1$ $\lambda$-a.e.~in $Q$.  Furthermore, as $\{\bm{\nu}_{k}\}_{1 \leq k \leq d}$ is an orthonormal basis, a simple calculation shows that $\left (\sum_{l=1}^{d} \bm{\nu}_{l} \otimes \bm{\nu}_{l} \right ) \bm{\zeta} = \bm{\zeta}$ for any $\bm{\zeta} \in \R^{d}$ with $\bm{\zeta} = \sum_{i=1}^{d} \zeta_{i} \bm{\nu}_{i}$.  Hence, $\sum_{k=1}^{d} \bm{\nu}_{k} \otimes \bm{\nu}_{k} = \id$ $\lambda$-a.e.~in $Q$.
\end{proof}	

\section{Passing to the limit}\label{sec:pass}
We recall the following strong convergences:
\begin{align*}
\varphi^{\eps_{j}} \to -1 + 2 \chara_{\Omega_{+}} & \text{ strongly in } C^{0}([0,T];L^{2}(\Omega)) \cap L^{2}(Q), \\
\theta^{\eps_{j}} \to \theta & \text{ strongly in } L^{2}(Q).
\end{align*}
Testing \eqref{CHD:div} with an arbitrary test function $\zeta \in C^{1}(\overline{Q})$ and integrating by parts, then passing to the limit leads to
\begin{align}\label{SIM:div}
0 = \int_{Q} \bm{v} \cdot \nabla \zeta + \zeta \HH \dx \dt.
\end{align}
Then, testing \eqref{CHD:varphi} expressed as
\begin{align*}
\pd_{t}(1 + \varphi^{\eps_{j}}) + \div (\varphi^{\eps_{j}} \bm{v}^{\eps_{j}}) = \div (m(\varphi^{\eps_{j}}) \nabla \mu^{\eps_{j}}) + \U \varphi^{\eps_{j}}
\end{align*} 
with an arbitrary test function $\zeta \in C^{1}(\overline{Q})$ such that $\zeta(T) = 0$, integrating in time, integrating by parts and passing to the limit, where we use the boundedness and continuity of $m(\cdot)$ to deduce the strong convergence of $m(\varphi^{\eps_{j}}) \nabla \zeta$ to $m(-1 + 2 \chara_{\Omega_{+}}) \nabla \zeta$ in $L^{2}(Q)^{d}$ via the dominated convergence theorem, leads to \eqref{CHD:SIM:mu}.
Similarly, testing \eqref{CHD:theta} expressed as
\begin{align*}
\pd_{t} ( \theta^{\eps_{j}} + \chi(1 + \varphi^{\eps_{j}})) + \div ( ( \theta^{\eps_{j}} + \chi \varphi^{\eps_{j}}) \bm{v}^{\eps_{j}}) = \div (n(\varphi^{\eps_{j}}) \nabla \theta^{\eps_{j}}) + \So
\end{align*}
with an arbitrary test function $\zeta \in C^{1}(\overline{Q})$ such that $\zeta(T) = 0$, integrating in time, integrating by parts and passing to the limit yields \eqref{CHD:SIM:theta}.
Meanwhile, testing \eqref{CHD:Darcy} with an arbitrary $\bm{Y} \in C^{0}([0,T];C^{1}_{0}(\Omega; \R^{d}))$, and upon integrating by parts yields
\begin{align*}
0 & = \int_{Q} \left ( \K \bm{v}^{\eps_{j}} + \nabla p^{\eps_{j}} - (\mu^{\eps_{j}} + \chi \theta^{\eps_{j}}) \nabla (1+\varphi^{\eps_{j}}) - \tfrac{\chi^{2}}{2}  \nabla \abs{\varphi^{\eps_{j}}}^{2} \right )\cdot \bm{Y} \dx \dt \\
& = \int_{Q} \K \bm{v}^{\eps_{j}} \cdot \bm{Y} - \left ( p^{\eps_{j}} - \tfrac{\chi^{2}}{2} \abs{\varphi^{\eps_{j}}}^{2} \right ) \div \bm{Y} + \div \left ( \left ( \mu^{\eps_{j}} + \chi \theta^{\eps_{j}} \right ) \bm{Y} \right ) (1 + \varphi^{\eps_{j}}) \dx \dt.
\end{align*}
Passing to the limit leads to \eqref{CHD:SIM:Darcy}
once we used $\bm{Y} = \bm{0}$ on $\pd \Omega$ and the divergence theorem to deduce that
\begin{align*}
\int_{Q} \abs{\varphi^{\eps_{j}}}^{2} \div \bm{Y} \dx \dt \to \int_{Q} \abs{-1 + 2 \chara_{\Omega_{+}(t)}}^{2} \div \bm{Y} \dx \dt = \int_{Q} \div \bm{Y} \dx \dt = 0.
\end{align*}
It remains to pass to the limit in the equation \eqref{CHD:mu} and construct the varifold.  Testing \eqref{CHD:mu} with $\bm{Y} \cdot \nabla (1+ \varphi^{\eps_{j}})$, where $\bm{Y} \in C^{0}([0,T];C^{1}_{0}(\Omega; \R^{d}))$ is arbitrary, leads to 
\begin{align*}
& \int_{Q} \left ( e^{\eps_{j}}(\varphi^{\eps_{j}}) - \eps_{j} \nabla \varphi^{\eps_{j}} \otimes \nabla \varphi^{\eps_{j}} \right )  \id : \nabla \bm{Y}  \dx \dt \\
& \quad = - \int_{Q} \bm{Y} \cdot \nabla (1+\varphi^{\eps_{j}}) \left ( \mu^{\eps_{j}} + \chi \theta^{\eps_{j}} - \chi^{2} \varphi^{\eps_{j}} \right ) \dx \dt \\
& \quad = \int_{Q}  (1 + \varphi^{\eps_{j}}) \div \left ( \left ( \mu^{\eps_{j}} + \chi \theta^{\eps_{j}} \right ) \bm{Y} \right ) + \frac{\chi^{2}}{2} \abs{\varphi^{\eps_{j}}}^{2} \div \bm{Y} \dx \dt.
\end{align*}
Passing to the limit $j \to \infty$ yields for the right-hand side
\begin{align*}
\int_{Q} 2 \chara_{\Omega_{+}(t)} \div \left ( \left ( \mu + \chi \theta \right ) \bm{Y} \right ) \dx \dt.
\end{align*}
The left-hand side can be handled in a similar fashion to \cite[\S 3.5]{Chen96}, and we obtain
\begin{align*}
\int_{Q} \nabla \bm{Y} : \left ( \id - \bm{\Upsilon} \right ) \dd \lambda^{t}(x) \dt & = \int_{Q} \nabla \bm{Y} : \left ( \id - \sum_{i=1}^{d} (\pi_{i} \bm{\nu}_{i} \otimes \bm{\nu}_{i})(x,t) \right ) \dd \lambda^{t}(x) \dt \\
& = \int_{Q} \nabla \bm{Y} : \sum_{i=1}^{d} c_{i}^{t}(x) \left ( \id - (\bm{\nu}_{i} \otimes \bm{\nu}_{i})(x,t) \right ) \dd \lambda^{t}(x) \dt,
\end{align*}
where $c_{i}^{t}(x) := \pi_{i}(x,t) + \frac{1}{d-1} \left ( 1 - \sum_{j=1}^{d} \pi_{j}(x,t) \right )$.  Indeed, a short calculation using $\sum_{i=1}^{d} \bm{\nu}_{i} \otimes \bm{\nu}_{i} = \id$ and $\sum_{i=1}^{d} \id = d \id$ shows that the second equality:
\begin{align*}
& \id - \sum_{i=1}^{d} \pi_{i} \bm{\nu}_{i} \otimes \bm{\nu}_{i} = \sum_{i=1}^{d} \pi_{i} \id - \sum_{i=1}^{d} \pi_{i} \bm{\nu}_{i} \otimes \bm{\nu}_{i} + \frac{d}{d-1} \left (1 - \sum_{j=1}^{d} \pi_{j} \right ) \id  - \frac{1}{d-1} \left (1 - \sum_{j=1}^{d} \pi_{j} \right ) \id \\
& \quad = \sum_{i=1}^{d} \pi_{i}(\id - \bm{\nu}_{i} \otimes \bm{\nu}_{i}) + \frac{1}{d-1} \left ( 1 - \sum_{j=1}^{d} \pi_{j} \right ) \left ( \sum_{i=1}^{d} \id - \bm{\nu}_{i} \otimes \bm{\nu}_{i} \right ) = \sum_{i=1}^{d} c^{t}_{i}( \id - \bm{\nu}_{i} \otimes \bm{\nu}_{i}).
\end{align*} 
The properties $0 \leq \pi_{i}$ and $\sum_{i=1}^{d} \pi_{i} \leq 1$ imply that $c^{t}_{i} \geq 0$, and similarly, $c^{t}_{i} = \frac{d-2}{d-1} \pi_{i} + \frac{1}{d-1}(1 - \sum_{j \neq i} \pi_{j}) \leq \frac{d-2}{d-1} + \frac{1}{d-1} = 1$ thanks to the property $0 \leq \pi_{i} \leq 1$.  This establishes the assertion that $0 \leq c^{t}_{i} \leq 1$.  Moreover, using $\sum_{i=1}^{d} \pi_{i} \leq 1$ we easily infer that $\sum_{i=1}^{d} c^{t}_{i} = \sum_{i=1}^{d} \pi_{i} + \frac{d}{d-1} \left ( 1 - \sum_{j=1}^{d} \pi_{j} \right ) = \frac{d}{d-1} - \frac{1}{d-1} \sum_{i=1}^{d} \pi_{i} \geq 1$.

For a.e.~$t \in (0,T)$ we define a Radon measure $V^{t}$ on $\Omega \times \PP^{d-1}$ by
\begin{align}\label{Varifold:Rep}
\dd V^{t}(x,P) := \sum_{i=1}^{d} c^{t}_{i}(x) \dd \lambda^{t}(x) \, \delta_{\bm{\nu}_{i}(t,x)}(P),
\end{align}
where $\delta_{\bm{\nu}}(P)$ is the projection onto the hyperplane normal to $\bm{\nu}$.  In particular, identifying $\PP^{d-1}$ with $\mathbb{S}^{d-1}/ \{ \nu, - \nu \}$, then $\delta_{\bm{\nu}}$ is just a Dirac measure at $\bm{\nu}$, i.e., $\delta_{\bm{\nu}}(A) = 1$ if $\bm{\nu} = A$ and $0$ if $\bm{\nu} \neq A$.  Then, we define the varifold $V = V^{t} \dt$ and see that \eqref{Varifold:Rep} satisfies the representation formula \eqref{Vari:represent}.  Furthermore, by \eqref{1stVar} we have
\begin{align*}
& \int_{Q} \nabla \bm{Y} : \sum_{i=1}^{d} c_{i}^{t}(x) \left ( \id - (\bm{\nu}_{i} \otimes \bm{\nu}_{i})(x,t) \right ) \dd \lambda^{t}(x) \dt \\
& \quad =  \int_{Q \times \PP^{d-1}} \nabla \bm{Y} : (\id - P \otimes P ) \dd V^{t}(x,P) \dt = \int_{0}^{T} \inner{\delta V^{t}}{\bm{Y}} \dt,
\end{align*}
which leads to \eqref{CHD:SIM:kappa}.

\section{Sharp interface limits for other variants}\label{sec:diss}

\subsection{Zero-velocity variant}
We can easily adapt the above proof to study the sharp interface limit of \eqref{CHD} with zero velocity:
\begin{subequations}\label{CH}
\begin{alignat}{3}
\pd_{t} \e{\varphi} & = \div (m(\e{\varphi}) \nabla \e{\mu}) + \U \e{\varphi} && \text{ in } Q, \label{CH:varphi} \\
\e{\mu} & = \eps^{-1} \Psi'(\e{\varphi}) - \eps \Laplace \e{\varphi} - \chi \e{\theta} - \chi^{2} \e{\varphi} && \text{ in } Q, \\
\pd_{t} \e{\theta} + \chi \pd_{t} \e{\varphi} & = \div (n(\e{\varphi}) \nabla \e{\theta}) + \So && \text{ in } Q \label{CH:theta}, \\
\pdnu \e{\varphi} & = \pdnu \e{\mu} = \pdnu \e{\theta} = 0 && \text{ on } \Sigma, \\
\e{\varphi}(0) & = \e{\varphi}_{0}, \quad \e{\theta}(0) = \e{\sigma}_{0} - \chi \e{\varphi}_{0} && \text{ in } \Omega,
\end{alignat}
\end{subequations}

The key difference between the analysis for \eqref{CHD} and \eqref{CH} is that we obtain uniform estimates for $\{\pd_{t} \e{\varphi} \}_{\eps \in (0,1]}$ and $\{\pd_{t} \e{\theta}\}_{\eps \in (0,1]}$ in the more regular space $L^{2}(0,T;H^{1}(\Omega)')$.  This is evident from the inspection of \eqref{CH:varphi} and \eqref{CH:theta} once the analogue of the first uniform estimate \eqref{CHD:unif:est:1} is derived.  With the better regularity, we can deduce that the limit $\theta$ of the (sub)sequence $\{\theta^{\eps_{j}}\}_{j \in \N}$ also belongs to $C^{0}([0,T];L^{2}(\Omega))$, and in particular the initial condition $\theta_{0}$ is attained a.e.~in $\Omega$ and as an equality in $L^{2}(\Omega)$, rather than as in \eqref{theta:ini:attain}.  The corresponding sharp interface limit of \eqref{CH} is simply \eqref{S:+}-\eqref{S:G3} with $\HH = 0$, $\bm{v} = \bm{0}$ and neglecting the condition $\jump{p} = \kappa$ on $\Gamma$, i.e., 
\begin{subequations}\label{CH:SIM}
\begin{alignat}{3}
- \div (m \nabla \mu) = \U , \quad \pd_{t} \theta + \div ( \theta \bm{v} - n \nabla \theta) = \So & \text{ in } \Omega_{+}, \\
- \div (m \nabla \mu) = - \U, \quad \pd_{t} \theta + \div ( \theta \bm{v} - n \nabla \theta) = \So & \text{ in } \Omega_{-}, \\
\jump{\mu} = 0, \quad \jump{\theta} = 0, \quad 2 \mu + 2 \chi \theta  = \kappa & \text{ on } \Gamma,   \\
-2\velo = \jump{m \nabla \mu} \cdot \bm{\nu}, \quad -2 \chi \velo  = \jump{n \nabla \theta} \cdot \bm{\nu} & \text{ on } \Gamma, \\
m \pdnu \mu = 0, \quad n \pdnu \theta = 0 & \text{ on } \pd \Omega \setminus \Gamma.
\end{alignat}
\end{subequations}
Furthermore, in this case we can take the exponent $q$ in \eqref{Psi''lowergrowth} to be $q > 2$ like in the original analysis of Chen \cite{Chen96} as oppose to $q \geq 4$ in the current analysis for the Cahn--Hilliard--Darcy system and the analysis for the Cahn--Hilliard--Brinkman system below.  We summarize this in the following theorem.

\begin{thm}[Sharp interface limit for the zero-velocity variant]\label{thm:CH}
For $\eps > 0$, let $(\e{\varphi}, \e{\mu}, \e{\theta})$ be a solution to \eqref{CH} with initial data $(\e{\varphi}_{0}, \e{\sigma}_{0})$.  Assume that \eqref{Psi''lowergrowth} holds with an exponent $q > 2$. Then, there exists a sequence $\{\eps_{j}\}_{j \in \N}$, $\eps_{j} \to 0$ as $j \to \infty$ such that properties $\mathrm{(1)}$, $\mathrm{(2)}$, $\mathrm{(3)}$ of Theorem \ref{thm:main} hold, as well as:
\begin{enumerate}
\item[$\mathrm{(4)}$] There exist functions $\mu \in L^{2}(0,T;H^{1}(\Omega))$, $\theta \in C^{0}([0,T];L^{2}(\Omega)) \cap L^{2}(0,T;H^{1}(\Omega))$ such that \eqref{mu:theta:L2H1}-\eqref{theta:L2L2} are satisfied with $\theta(0) = \theta_{0} := \sigma_{0} + \chi - 2 \chi \chara_{\Omega_{+}(0)}$ in $L^{2}(\Omega)$.
\item[$\mathrm{(5)}$] The triplet $(V, \mu, \theta)$ is a varifold solution to \eqref{CH:SIM} with initial values $(\chara_{\Omega_{+}(0)}, \theta_{0})$ in the sense that \eqref{CHD:SIM:mu}-\eqref{CHD:SIM:kappa} and \eqref{CH:SIM:energy} with $\bm{v} = \bm{0}$, $\HH = 0$ are satisfied.
\end{enumerate}

\end{thm}

\subsection{Solenoidal Brinkman variant}
We consider replacing the non-solenoidal Darcy law with a solenoidal Brinkman law, that is, consider \eqref{CHD:varphi}-\eqref{CHD:ini} with
\begin{subequations}\label{CHB}
\begin{alignat}{3}
\div \e{\bm{v}} & = 0 && \text{ in } Q, \\
-\div (2 \eta \der \e{\bm{v}}) + \K \e{\bm{v}} & = - \nabla \e{p} + (\e{\mu} + \chi \e{\theta} + \chi^{2} \e{\varphi}) \nabla \e{\varphi} &&\text{ in } Q, \label{CHB:Brink} \\
\e{\bm{v}} & = \bm{0} && \text{ on } \Sigma, \label{CHB:bc}
\end{alignat}
\end{subequations}
where $\der \bm{v} := \tfrac{1}{2} (\nabla \bm{v} + (\nabla \bm{v})^{\top})$ and $\eta > 0$ is a fixed constant.  The corresponding sharp interface limit is
\begin{subequations}\label{CHB:SIM}
\begin{alignat}{3}
\div \bm{v} = 0 , \quad - \div (2 \eta \der \bm{v}) + \K \bm{v} = - \nabla p  & \text{ in } \Omega_{+} \cup \Omega_{-}, \label{weak:SIM:B1}  \\
- \div (m \nabla \mu) = \U - \HH, \quad \pd_{t} \theta + \div ( \theta \bm{v} - n \nabla \theta) = \So - \chi \HH & \text{ in } \Omega_{+}, \\
- \div (m \nabla \mu) = - \U + \HH, \quad \pd_{t} \theta + \div ( \theta \bm{v} - n \nabla \theta) = \So + \chi \HH & \text{ in } \Omega_{-},  \\
2 \mu + 2 \chi \theta  = \kappa, \quad \jump{\mu} = 0, \quad \jump{\theta} = 0 & \text{ on } \Gamma, \\
\jump{\bm{v}} = \bm{0}, \quad \jump{p}\bm{\nu} - 2 \eta \jump{\der \bm{v}} \bm{\nu}= \kappa \bm{\nu} & \text{ on } \Gamma, \label{weak:SIM:B2} \\
2(-\velo + \bm{v} \cdot \bm{\nu}) = \jump{m \nabla \mu} \cdot \bm{\nu}, \quad 2 \chi (-\velo + \bm{v} \cdot \bm{\nu}) = \jump{n \nabla \theta} \cdot \bm{\nu} & \text{ on } \Gamma,  \\
m \pdnu \mu = 0, \quad n \pdnu \theta = 0, \quad \bm{v} = \bm{0} & \text{ on } \pd \Omega \setminus \Gamma.
\end{alignat}
\end{subequations}
In particular, the velocity $\bm{v}$ is continuous across the interface, as dictated by the condition $\jump{\bm{v}} = \bm{0}$, the Darcy law $\K \bm{v} = - \nabla p$ in \eqref{S:div} is replaced by the Brinkman law $- \div ( 2 \eta \der \bm{v}) + \K \bm{v} = - \nabla p$ and the Young--Laplace law $\jump{p} = \kappa$ in \eqref{S:G1} is replaced by the stress balance $ \jump{p} \bm{\nu} - 2 \eta \jump{ \der \bm{v}} \bm{\nu} = \kappa \bm{\nu}$.

Let us point out that for the non-solenoidal case $\div \e{\bm{v}} = \HH$, if we prescribe that $\HH$ belongs to the space $L^{\infty}(0,T; H^2_n \cap L^{\infty}_{0})$, then we infer that the pressure $\e{p}$ satisfies the Poisson equation
\begin{align*}
- \Laplace \e{p} = \K \HH - 2 \eta \Laplace \HH - \div \left ( \left ( \e{\mu} + \chi \e{\theta} + \chi^{2} \e{\varphi} \right ) \nabla \e{\varphi} \right ) \text{ in } \Omega.
\end{align*}
For the calculations in Section \ref{sec:FirstEst} in dealing with the term $\HH (\e{p} - \e{\mu} \e{\varphi})$ appearing on the right-hand side of \eqref{CHD:energy:id}, it is desirable to have a homogeneous Neumann boundary $\pdnu \e{p} = 0$ on $\pd \Omega$ for the pressure, so that we can also set $\mean{\e{p}} = 0$, and write
\begin{align*}
\e{p} = \NN \left (\K \HH - 2 \eta \Laplace \HH - \div \left ( \left ( \e{\mu} - \mean{\e{\mu}} + \chi \e{\theta} + \chi^{2} \e{\varphi} \right ) \nabla \e{\varphi} \right ) \right ) + \mean{\e{\mu}} (\e{\varphi} - \mean{\e{\varphi}}).
\end{align*}
However, in doing so we then introduce additional boundary terms of the form
\begin{align*}
\int_{\pd \Omega} 2 \eta (\der \e{\bm{v}}) \bm{\nu} \cdot \e{\bm{v}} - \e{p} \e{\bm{v}} \cdot \bm{\nu},
\end{align*} 
to the energy identity \eqref{CHD:energy:id} when we test the Brinkman equation \eqref{CHB:Brink} with $\e{\bm{v}}$, and this boundary term seems not to vanish if we prescribe $\pdnu \e{p} = \K \e{\bm{v}} \cdot \bm{\nu} - 2 \eta \Laplace \e{\bm{v}} \cdot \bm{\nu} = 0$ on $\pd \Omega$.

Hence, we consider only the solenoidal case \eqref{CHB} and the sharp interface analysis simplifies considerably.  The weak formulation for the Brinkman subsystem \eqref{CHB} is
\begin{align}\label{weak:Brink}
\int_{\Omega} 2 \eta \der \e{\bm{v}}(t) : \der \bm{\zeta} + \K \e{\bm{v}}(t) \cdot \bm{\zeta} - (\e{\mu}(t) + \chi \e{\theta}(t) + \chi^{2} \e{\varphi}(t)) \nabla \e{\varphi}(t) \cdot \bm{\zeta} \dx = 0
\end{align}
for a.e.~$t \in (0,T)$ and for all $\bm{\zeta} \in \bm{H}^{1}_{0,\sigma}(\Omega)$, where the space $\bm{H}^{1}_{0,\sigma}(\Omega)$ is the completion of $C^{\infty}_{0,\sigma}(\Omega) := \{ \bm{f} \in C^{\infty}_{0}(\Omega)^{d} \, : \, \div \bm{f} = 0 \}$ with respect to the $H^{1}(\Omega)^{d}$ norm, and the pressure is eliminated.  Testing \eqref{CHD:varphi} with $\e{\mu}$, \eqref{CHD:mu} with $\pd_{t} \e{\varphi}$, \eqref{CHD:theta} with $\e{\theta}$, and \eqref{weak:Brink} with $\e{\bm{v}}$ leads to the energy identity
\begin{align*}
& \EE(\e{\varphi}(t), \e{\theta}(t)) + \int_{Q_{t}} m(\e{\varphi}) \abs{\nabla \e{\mu}}^{2} + n(\e{\varphi}) \abs{\nabla \e{\theta}}^{2} + \K \abs{\e{\bm{v}}}^{2} + 2 \eta \abs{\der \e{\bm{v}}}^{2} \dx \ds \\
& \quad = \int_{Q_{t}} \U \e{\varphi} \e{\mu} + \So \e{\theta} \dx \ds + \EE(\e{\varphi}_{0}, \e{\theta}_{0}),
\end{align*}
where $\mathcal{E}(\e{\varphi}, \e{\theta})$ is given in \eqref{CH:EE:form}.  Then, from this energy estimate we can derive the uniform estimates \eqref{CHD:unif:est:1}-\eqref{CHD:unif:est:3} and \eqref{CHD:unif:est:7} with additionally the control $\norm{\eta^{1/2} \der \e{\bm{v}}}_{L^{2}(Q)} \leq C$, so that by Korn's inequality (recall the boundary condition \eqref{CHB:bc}) we infer that $\e{\bm{v}}$ is bounded in $L^{2}(0,T;H^{1}(\Omega)^{d})$.  

Furthermore, employing the boundedness of $\e{\bm{v}}$ in $L^{2}(0,T;L^{6}(\Omega)^{d})$, we find that $\div (\e{\varphi} \e{\bm{v}})$ and $\pd_{t} \e{\varphi}$ are bounded in $L^{2}(0,T;H^{1}(\Omega)')$.  By the interpolation inequality $\norm{f}_{L^{3}} \leq \norm{f}_{L^{2}}^{1/2} \norm{f}_{L^{6}}^{1/2}$ and the boundedness of $\e{\theta}$ in $L^{\infty}(0,T;L^{2}(\Omega)) \cap L^{2}(0,T;L^{6}(\Omega))$, we find that $\div(\e{\theta} \e{\bm{v}})$ and $\pd_{t} \e{\theta}$ are bounded in $L^{\frac{4}{3}}(0,T;H^{1}(\Omega)')$.  Then, the H\"{o}lder-in-time uniform estimates \eqref{CHD:unif:est:5}-\eqref{CHD:unif:est:6} follow from before.

We can derive estimates on a pressure variable, abusing notation and denoted as $\e{p}$ again, by standard arguments in the study of solenoidal fluid equations.  From the weak formulation \eqref{weak:Brink}, we define a distribution $\e{\bm{F}}$ as
\begin{align*}
\inner{\e{\bm{F}}}{\bm{\zeta}} & = \int_{\Omega} 2 \eta \der \e{\bm{v}}(t) : \der \bm{\zeta} + \K \e{\bm{v}}(t) \cdot \bm{\zeta} + \e{\varphi}(t) \nabla (\e{\mu}(t) + \chi \e{\theta}(t)) \cdot \bm{\zeta} \dx \\
& \quad + \int_{\Omega} \e{\varphi}(t)(\e{\mu}(t) + \chi \e{\theta}(t)) \div \bm{\zeta} + \tfrac{1}{2} \chi^{2} \abs{\e{\varphi}(t)}^{2} \div \bm{\zeta} \dx
\end{align*}
for $\bm{\zeta} \in \bm{H}^{1}_{0,\sigma}(\Omega)$.  Denoting by $\bm{H}^{-1}$ the dual space of $\bm{H}^{1}_{0,\sigma}(\Omega)$, and employing the uniform estimates obtained above, we find that $\e{\bm{F}}$ is bounded in $L^{2}(0,T;\bm{H}^{-1})$ (cf. similar calculations in \eqref{pressure:est}) and $\e{\bm{F}}$ vanishes on the subspace $C^{\infty}_{0}([0,T];C^{\infty}_{0,\sigma}(\Omega))$ by \eqref{weak:Brink}.  Applying for example \cite[\S~IV, Lem.~1.4.1]{Sohr} shows that there exists a unique $\e{p} \in L^{2}(0,T;L^{2}_{0}(\Omega))$ satisfying $\e{\bm{F}} = - \nabla \e{p}$ in the sense of distributions for a.e.~$t \in (0,T)$ and 
\begin{align}\label{CHB:press:est}
\norm{\e{p}}_{L^{2}(Q)} \leq C \norm{\e{\bm{F}}}_{L^{2}(0,T;\bm{H}^{-1})} \leq C
\end{align}
for positive constants $C$ independent of $\eps$.  We formulate the sharp interface limit for the solenoidal Cahn--Hilliard--Brinkman system in the following theorem.

\begin{thm}[Sharp interface limit for the solenoidal Brinkman system]\label{thm:CHB}
For $\eps > 0$, let $(\e{\varphi}, \e{\mu}, \e{\theta}, \e{\bm{v}})$ be a solution to \eqref{CHD:varphi}-\eqref{CHD:ini}, \eqref{weak:Brink} with initial data $(\e{\varphi}_{0}, \e{\sigma}_{0})$.  Then, there exists a sequence $\{\eps_{j}\}_{j \in \N}$, $\eps_{j} \to 0$ as $j \to \infty$ such that properties $\mathrm{(1)}$, $\mathrm{(2)}$, $\mathrm{(3)}$ of Theorem \ref{thm:main} hold, as well as:
\begin{enumerate}
\item[$\mathrm{(4)}$] There exist functions $\mu \in L^{2}(0,T;H^{1}(\Omega))$, $\theta \in C^{0}_{\mathrm{w}}([0,T];L^{2}(\Omega)) \cap L^{2}(0,T;H^{1}(\Omega))$, $\bm{v} \in L^{2}(0,T;\bm{H}^{1}_{0,\sigma}(\Omega))$, $p \in L^{2}(0,T;L^{2}_{0}(\Omega))$ such that \eqref{mu:theta:L2H1}-\eqref{theta:L2L2}, \eqref{p:L2L2}, \eqref{theta:ini:attain} and
\begin{align*}
\e{\bm{v}} & \longrightarrow \bm{v} \text{ weakly in } L^{2}(0,T;H^{1}(\Omega)^{d})
\end{align*}
are satisfied.
\item[$\mathrm{(5)}$] The quadruple $(V, \mu, \theta, \bm{v})$ is a varifold solution to \eqref{CHB:SIM} with initial values $(\chara_{\Omega_{+}(0)}, \theta_{0})$ in the sense that \eqref{CHD:SIM:div}, \eqref{CHD:SIM:mu}-\eqref{CHD:SIM:kappa} are satisfied with $\HH = 0$, and
\begin{align}\label{CHB:SIM:Brink}
0 = \int_{Q} 2 \eta \der \bm{v} : \der \bm{Y} + \K \bm{v} \cdot \bm{Y} - p \div \bm{Y} + 2 \chara_{\Omega_{+}} \div \left ( \left ( \mu + \chi \theta \right ) \bm{Y} \right ) \dx \dt,
\end{align}
holds for all $\bm{Y} \in C^{0}([0,T];C^{\infty}_{0}(\overline{\Omega}; \R^{d}))$.  Furthermore, for a.e.~$0 \leq \tau < t \leq T$ and $\varphi(x,s) = -1 + 2 \chara_{\Omega_{+}(t)}(x)$, it holds that 
\begin{equation}\label{CHB:SIM:energy}
\begin{aligned}
& \lambda^{t}(\Omega) + \frac{1}{2} \norm{\theta(t)}_{L^{2}}^{2} + \int_{\tau}^{t} \int_{\Omega} m(\varphi) \abs{\nabla \mu}^{2} + n(\varphi) \abs{\nabla \theta}^{2} + \K \abs{\bm{v}}^{2} + 2 \eta \abs{\der \bm{v}}^{2} \dx \ds \\
& \quad \leq \lambda^{\tau}(\Omega) +  \frac{1}{2} \norm{\theta(\tau)}_{L^{2}}^{2} + \int_{\tau}^{t} \int_{\Omega} \U \varphi \mu + \So \theta  \dx \ds.
\end{aligned}
\end{equation}
\end{enumerate}
\end{thm}
\begin{proof}
The sharp interface energy inequality \eqref{CHB:SIM:energy} follows from a similar argument as in the proof of Lemma \ref{lem:meas+energyIneq}.  It remains to derive \eqref{CHB:SIM:Brink}, which can be viewed as the weak formulation of the Brinkman system \eqref{weak:SIM:B1}, \eqref{weak:SIM:B2}.  Returning to \eqref{weak:Brink}, integrating by parts, and using the recovery of a pressure $\e{p}$ yields
\begin{equation}\label{CHB:eps:Brinkman:weak}
\begin{aligned}
& \inner{\e{\bm{F}}}{\bm{Y}} = \inner{-\nabla \e{p}}{\bm{Y}}\\
\Rightarrow & \int_{Q} 2 \eta \der \e{\bm{v}} : \der \bm{Y} + \K \e{\bm{v}} \cdot \bm{Y} + (1 + \e{\varphi}) \div ((\e{\mu} + \chi \e{\theta}) \bm{Y}) + \tfrac{\chi^{2}}{2} \abs{\e{\varphi}}^{2} \div \bm{Y} \dx \dt \\
& \quad  = \int_{Q} \e{p} \div \bm{Y} \dx \dt
\end{aligned}
\end{equation}
for all $\bm{Y} \in C^{0}([0,T];C^{\infty}_{0}(\overline{\Omega};\R^{d}))$.  Passing to the limit along a subsequence $\{\eps_{j}\}_{j \in \N}$, $\eps_{j} \to 0$ as $j \to \infty$ and employing the convergences properties outlined in point (4) yields \eqref{CHB:SIM:Brink}.
\end{proof}

We point out that one may also consider scaling the viscosity $\eta$, which was a fixed constant, with $\eps$, i.e., $\eta = \eps^{\beta}$ for some $\beta > 0$.  It turns out that for any $\beta \in \N$, a formally matched asymptotic analysis shows that the sharp interface limit for the Brinkman variant with $\eta = \eps^{\beta}$ is \eqref{CHD:SIM}, the sharp interface limit of the Darcy variant (with $\HH = 0$).  However, with the above compactness approach we can prove this for any $\beta \in \R_{> 0}$ due to the following observations:

\begin{enumerate}
\item[(1)] Due to the choice $\eta = \eps^{\beta}$, $\norm{\e{\bm{v}}}_{L^{2}(Q)}$ and $\norm{\eps^{\frac{\beta}{2}} \der \e{\bm{v}}}_{L^{2}(Q)}$ are uniformly bounded.  
\item[(2)] We still obtain $\norm{\e{p}}_{L^{2}(Q)} \leq C$ by following an analogous computation to \eqref{CHB:press:est}.
\item[(3)] When passing to the limit in \eqref{CHB:eps:Brinkman:weak}, thanks to $\eps^{\frac{\beta}{2}} \norm{\der \e{\bm{v}}}_{L^{2}(Q)} \leq C$ and
\begin{align*}
\abs{ \int_{Q} \eps^{\beta} \der \e{\bm{v}} : \der \bm{Y} \dx \ds} \leq C \eps^{\frac{\beta}{2}} \norm{\der \bm{Y}}_{L^{2}(Q)} \to 0,
\end{align*}
we obtain \eqref{CHD:SIM:Darcy}.
\item[(4)] There exists $\bm{\xi} \in L^{2}(Q)^{d \times d}$ such that $\eps^{\frac{\beta}{2}} \der \bm{v}^{\eps_{j}} \to \bm{\xi}$ weakly in $L^{2}(Q)^{d \times d}$.  Passing to the limit in the associated energy identity yields \eqref{CH:SIM:energy} with $\HH = 0$ and an additional non-negative term $\int_{\tau}^{t} \int_{\Omega} 2 \abs{\bm{\xi}}^{2} \dx \ds$ on the left-hand side, which we can subsequently neglect and as a result recover \eqref{CH:SIM:energy} with $\HH = 0$.
\end{enumerate}
\begin{thm}[Alternate sharp interface limit for the solenoidal Brinkman system]\label{thm:CHB:Alt}
Let $\beta \in \R_{> 0}$ and set $\eta = \eps^{\beta}$.  For $\eps > 0$, let $(\e{\varphi}, \e{\mu}, \e{\theta}, \e{\bm{v}})$ be a solution to \eqref{CHD:varphi}-\eqref{CHD:ini}, \eqref{weak:Brink} with initial data $(\e{\varphi}_{0}, \e{\sigma}_{0})$.  Then, there exists a sequence $\{\eps_{j}\}_{j \in \N}$, $\eps_{j} \to 0$ as $j \to \infty$ such that properties $\mathrm{(1)}$, $\mathrm{(2)}$, $\mathrm{(3)}$, $\mathrm{(4)}$ of Theorem \ref{thm:main} hold, as well as:
\begin{enumerate}
\item[$\mathrm{(5)}$] The quadruple $(V, \mu, \theta, \bm{v})$ is a varifold solution to \eqref{CHD:SIM} with $\HH = 0$ and initial values $(\chara_{\Omega_{+}(0)}, \theta_{0})$ in the sense that \eqref{CHD:SIM:div}-\eqref{CHD:SIM:kappa} and \eqref{CH:SIM:energy} are satisfied with $\HH = 0$.
\end{enumerate}
\end{thm}

In particular, the above result can be seen as a ``sharp interface'' analogue of \cite[Thm.~2.11]{BCG}, which asserts that the solutions to the Cahn--Hilliard--Brinkman model (with $\U = 0$ and neglecting the equation for $\theta$) converge to solutions to the Cahn--Hilliard--Darcy model (with $\HH = \U = 0$ and neglecting the equation for $\theta$) as $\eta \to 0$.

\section*{Acknowledgements}
The author would like to thank Johannes Daube for the discussion on the proofs of Lemmas \ref{lem:Psigeqabsvarphi-1square} and \ref{lem:ModicaAnsatz:BijW}, and Johannes Kampmann for proofreading the manuscript as well as the numerous valuable discussions on varifolds.  The support of a Direct Grant of CUHK (project 4053335) is gratefully acknowledged.


\begin{thebibliography}{99}

\bibitem{AGG}
H.~Abels, H.~Garcke and G.~Gr\"{u}n. {\em Thermodynamically consistent, frame indifferent diffuse interface models for incompressible two-phase flows with different densities}. Math. Models Methods Appl. Sci., \textbf{22} (2012), 1150013 (40 pages).

\bibitem{AK}
H.~Abels and J.~Kampmann. {\em On the sharp interface limit of a model for phase separation on biological membranes}.  Preprint arXiv:1811.12489 [math.AP] (2019)


\bibitem{AbelsLengeler}
H.~Abels and D.~Lengeler. {\em On sharp interface limits for diffuse interface models for two-phase flows}. Interfaces Free Bound., \textbf{16} (2014), 395--418.


\bibitem{AbelsSchaubeck1}
H.~Abels and S.~Schaubeck. {\em Sharp interface limit for the Cahn--Larch\'{e} system}. Asymptot. Anal., \textbf{91} (2015), 283--240.

\bibitem{AbelsSchaubeck2}
H.~Abels and S.~Schaubeck. {\em Nonconvergence of the Capillary Stress Functional for Solutions of the
Convective Cahn-Hilliard Equation}. In: Y.~Shibata and Y.~Suzuki, editors, \emph{Mathematical Fluid Dynamics, Present and Future}, Springer Japan, 2016, 3--23.

\bibitem{AbelsRoger}
H.~Abels and M.~R\"{o}ger. {\em Existence of weak solutions for a non-classical sharp interface model for a two-phase flow of viscous, incompressible fluids}. Ann. Inst. H. Poincar\'{e} Anal. Non Lin\'{e}aire, \textbf{26} (2009), 2403--2424.

\bibitem{Adams} 
R.A.~Adams and J.J.F.~Fournier. {\em Sobolev spaces}. Pure and applied mathematics, 2nd ed., Elsevier, Amsterdam, Boston, Heidelberg (2003).

\bibitem{Alikakos}
N.D.~Alikakos. {\em $L^{p}$ Bounds of solutions of reaction-diffusion equations}. Commun. Partial Differential Equations, \textbf{4} (1979), 827--868.

\bibitem{ABC}
N.D.~Alikakos, P.W.~Bates and X.~Chen. {\em The convergence of solutions of the Cahn--Hilliard equation to the solution of Hele--Shaw model}. Arch. Rational Mech. Anal., \textbf{128} (1994), 165--205.

\bibitem{Ambrosio}
L.~Ambrosio, N.~Fusco and D.~Pallara. {\em Functions of Bounded Variation and Free Discontinuity Problems}. Oxford Math. Monogr., Clarendon Press, Oxford (2000).

\bibitem{Anderson}
D.M.~Anderson, G.B.~McFadden and A.A.~Wheeler. {\em Diffuse-interface methods in fluid mechanics}. Annu. Rev. Fluid Mech., \textbf{30} (1998), 139--165.

\bibitem{Bertozzi}
A.L.~Bertozzi, S.~Esedo\={g}lu and A.~Gillette. {\em Inpainting of binary
images using the Cahn--Hilliard equation}. IEEE Trans. Image Process., \textbf{16} (2007), 285--291.

\bibitem{Brinkman}
H.C.~Brinkman. {\em A calculation of the viscous force exerted by a flowing fluid on a dense swarm of particles}. Appl. Sci. Res., \textbf{1} (1947), 27--34.

\bibitem{BCG} 
S.~Bosia, M.~Conti and M.~Grasselli. {\em On the Cahn--Hilliard--Brinkman system}. Commun. Math. Sci., \textbf{13} (2015), 1541--1567.

\bibitem{CCO}
E.A.~Carlen, M.C.~Carvalho and E.~Orlandi. {\em Approximate solutions of the Cahn--Hilliard equation via corrections to the Mullins--Sekerka motion}.  Arch. Rational Mech. Anal., \textbf{178} (2005), 1--55.


\bibitem{ChenSpectral}
X.~Chen. {\em Spectrum for the Allen--Cahn, Cahn--Hilliard, and phase--field equations for generic interfaces}. Comm. Partial Differential Equations, \textbf{19} (1994), 1371--1395.

\bibitem{Chen96} 
X.~Chen. {\em Global asymptotic limit of solutions of the Cahn--Hilliard equation}. J. Differential Geometry, \textbf{44} (1996), 262--311.

\bibitem{ChenLowen}
Y.~Chen, S.M.~Wise, V.B.~Shenoy and J.S.~Lowengrub. {\em A stable scheme for a nonlinear, multiphase tumor growth model with an elastic membrane}. Int. J. Numer. Meth. Biomed. Engng., \textbf{30} (2014), 726--754.


\bibitem{Cowan}
C.~Cowan. {\em The Cahn--Hilliard equation as a gradient flow}.  Master thesis, Simon Fraser University (2005), \url{http://summit.sfu.ca/system/files/iritems1/10196/etd1961.pdf}


\bibitem{CLLW} 
V.~Cristini, X.~Li, J.S.~Lowengrub and S.M.~Wise. {\em Nonlinear simulations of solid tumor growth using a mixture model: invasion and branching}. J. Math. Biol., \textbf{58} (2009),723--763.

\bibitem{CL_book}
V.~Cristini and J.~Lowengrub. {\em Multiscale Modeling of Cancer: An Integrated Experimental and Mathematical Modeling Approach}. Cambridge University Press (2010).

\bibitem{DGL}
L.~Ded\`{e}, H.~Garcke and K.F.~Lam. {\em A Hele--Shaw--Cahn--Hilliard model for incompressible two-phase flows with different densities}. J. Math. Fluid Mech., \textbf{20} (2018), 531--567.

\bibitem{Daube}
J.~Daube. {\em Sharp-interface limit for the Navier--Stokes--Korteweg equations}.  Ph.D. thesis, Universit\"{a}t Freiburg (2017) \url{https://freidok.uni-freiburg.de/data/11679}.

\bibitem{Droniou}
J.~Droniou. {\em Int\'{e}gration et Espaces de Sobolev \`{a} Valeurs Vectorielles}. Universit\'{e} de Provence, \url{http://http://www-gm3.univ-mrs.fr/polys/gm3-02/gm3-02.pdf} (2001).


\bibitem{Fei}
M.~Fei. {\em Global sharp interface limit of the Hele--Shaw--Cahn--Hilliard system}. Math. Meth. Appl. Sci., \textbf{40} (2017), 833--852.

\bibitem{Fei2}
M.~Fei, T.~Tao and W.~Wang. {\em Sharp interface limit of a diffuse interface model for tumor-growth}. Preprint arXiv:1708.07468.

\bibitem{Feireisl}
E.~Feireisl and A.~Novotn\'{y}. {\em Singular Limits in Thermodynamics of Viscous Fluids}. Advances in Mathematical Fluid Mechanics, Birkh\"{a}user Basel (2009).
 

\bibitem{FengCH}
X.~Feng and A.~Prohl. {\em Numerical analysis of the Cahn--Hilliard equation and approximation for the Hele--Shaw problem}. Interfaces Free Bound., \textbf{7} (2005), 1--28.

\bibitem{Fife}
P.C.~Fife and O.~Penrose. {\em Interfacial Dynamics for Thermodynamically Consistent Phase-–Field Models with Nonconserved Order Parameter}.  Electron. J. Differential Equations \textbf{1995} (1995), 1--49.

\bibitem{Frieboes}
H.B.~Frieboes, F.~Jin, Y.-L.~Chuang, S.M.~Wise, J.S.~Lowengrub and V.~Cristini. {\em Three-dimensional multispecies nonlinear tumor growth - II: Tumor invasion and angiogenesis}. J. Theor. Biol., \textbf{264} (2010), 1254--1278.


\bibitem{Garcke}
H.~Garcke. {\em On Cahn--Hilliard systems with elasticity}. Proc. Roy. Soc. Edinburgh Sect. A, \textbf{133}, (2003), 307--331.

\bibitem{GarckeKwak}
H.~Garcke and D.J.C.~Kwak. {\em On asymtptotic limits of Cahn--Hilliard systems with elastic misfit}. In: A.~Mielke, editor, \emph{Contribution in Analysis, Modeling and Simulation of Multiscale Problems}, Springer-Verlag, Berlin (2006), 87--112.


\bibitem{GLDarcy}
H.~Garcke and K.F.~Lam. {\em Global weak solutions and asymptotic limits of a Cahn--Hilliard--Darcy system modelling tumour growth}. AIMS Mathematics, \textbf{1} (2016), 318--360.


\bibitem{GLNeu}
H.~Garcke and K.F.~Lam. {\em Well-posedness of a Cahn--Hilliard system modelling tumour growth with chemotaxis and active transport}. European J. Appl. Math., \textbf{28} (2017), 284--316.


\bibitem{GLSS} 
H.~Garcke, K.F.~Lam, E.~Sitka and V.~Styles. {\em A Cahn--Hilliard--Darcy model for tumour growth with chemotaxis and active transport}. Math. Models Methods Appl. Sci., \textbf{26} (2016), 1095--1148.


\bibitem{JWZ} 
J.~Jiang, H.~Wu and S.~Zheng. {\em Well-posedness and long-time behavior of a non-autonomous Cahn--Hilliard--Darcy system with mass source modeling tumor growth}. J. Differential Equ., \textbf{259} (2015), 3032--3077.

\bibitem{Kim}
J.~Kim, S.~Lee, Y.~Choi, S-M.~Lee and D.~Jeong. {\em Basic principles and practical applications of the Cahn--Hilliard equation}. Math. Probl. Eng., \textbf{2016} (2016), Article ID 9532608, 11 pages., \url{http://dx.doi.org/10.1155/2016/9532608}

\bibitem{Kwak}
D.J.C.~Kwak. {\em The sharp-interface limit of the Cahn--Hilliard system with elasticity}. Ph.D. thesis, Universit\"{a}t Regensburg (2008) \url{https://epub.uni-regensburg.de/10700/}


\bibitem{Lang}
S.~Lang. {\em Real and Functional Analysis}. Graduate Texts in Mathematics, 3rd ed., Springer-Verlag New York (1993).

\bibitem{LarcheCahn} 
F.C.~Larch\'{e} and J.W.~Cahn. {\em The effects of self-stress on diffusion in solids}. Acta Metall., \textbf{30} (1982), 1835--1845.

\bibitem{Le}
N.Q.~Le. {\em A Gamma-convergence approach to the Cahn--Hilliard equation}. Calc. Var. Partial Differential Equations, \textbf{32} (2008), 499--522.

\bibitem{LLG1}
H.~Lee, J.~Lowengrub and J.~Goodman. {\em Modeling pinchoff and reconnection in a Hele--Shaw cell. I. The models and their calibration}. Phys. Fluids, \textbf{14} (2002), 492--513.

\bibitem{LLG2}
H.~Lee, J.~Lowengrub and J.~Goodman. {\em Modeling pinchoff and reconnection in a Hele--Shaw cell. II. Analysis and simulation in the nonlinear regime}. Phys. Fluids, \textbf{14} (2002), 514--545.

\bibitem{LTZ} 
J.S.~Lowengrub, E.~Titi and K.~Zhao. {\em Analysis of a mixture model of tumor growth}. European J. Appl. Math., \textbf{24} (2013), 691--734.

\bibitem{LT}
J.~Lowengrub and L.~Truskinovsky. {\em Quasi-incompressible Cahn--Hilliard fluids and topological transitions}. R. Soc. Lond. Proc. Ser. A Math. Phys. Eng. Sci., \textbf{454} (1998), 2617--2654.


\bibitem{Modica}
L.~Modica. {\em The gradient theory of phase transitions and the minimal interface criterion}. Arch. Rational Mech. Anal., \textbf{98} (1987), 123--142.

\bibitem{RM}
S.~Melchioma and E.~Rocca. {\em Varifold solutions of a sharp interface limit of a diffuse interface model for tumor growth}. Interface Free Bound., \textbf{19} (2018), 571–-590.

\bibitem{Pego}
R.L.~Pego. {\em Front migration in the nonlinear Cahn--Hilliard equation}. Proc. Roy. Soc. London A, \textbf{422} (1989), 261--278.

\bibitem{RoccaScala}
E.~Rocca and R.~Scala. {\em A rigorous sharp interface limit of a diffuse interface model related to tumor growth}. J. Nonlinear Sci., \textbf{27} (2017), 847--872.

\bibitem{SS}
E.~Sandier and S.~Serfaty. {\em Gamma-convergence of gradient flows with application to Ginzburg--Landau}.  Commun. Pure Appl. Math., \textbf{57} (2004), 1627--1672.

\bibitem{Simon86}
J.~Simon. {\em Compact sets in space $L^{p}(0,T;B)$}. Ann. Mat. Pura Appl., \textbf{146} (1987), 65--96.

\bibitem{Sohr}
H.~Sohr. {\em The Navier-–Stokes Equations: An Elementary Functional Analytic Approach},
Birkh\"{a}user Advanced Texts, Springer, Basel (2010).

\bibitem{Stoth}
B.E.E.~Stoth. {\em Convergence of the Cahn--Hilliard equation to the Mullins--Sekerka problem in spherical symmetry}. J. Differential Equations, \textbf{125} (1996), 154--183.

\bibitem{Wise}
S.M.~Wise, J.S.~Lowengrub, H.B.~Frieboes and V.~Cristini. {\em Three-dimensional multispecies nonlinear tumor growth - I: Model and numerical method}. J. Theor. Biol., \textbf{253} (2008), 524--543.

\bibitem{Zhu}
J.~Zhu, L.-Q.~Chen and J.~Shen. {\em Morphological evolution during phase separation and coarsening with strong inhomogeneous elasticity}. Modelling Simul. Mater. Sci. Eng., \textbf{9} (2011), 499--511.

\end{thebibliography}
\end{document}